\documentclass[10pt]{amsart}

\usepackage{graphicx}        
\usepackage{multicol}        
\usepackage[bottom]{footmisc}
\usepackage{amscd,amsmath,amssymb,amsfonts}
\usepackage[cmtip, all]{xy}

\newtheorem{theorem}{Theorem}[section]
\newtheorem{proposition}[theorem]{Proposition}
\newtheorem{lemma}[theorem]{Lemma}
\newtheorem{corollary}[theorem]{Corollary}
\newtheorem{conjecture}[theorem]{Conjecture}

\theoremstyle{definition}
\newtheorem{definition}[theorem]{Definition}
\theoremstyle{remark}
\newtheorem{remark}[theorem]{Remark}

\newtheorem{ex}[theorem]{Example}

\numberwithin{equation}{section}

\newcommand{\cl}{{\rm cl}}

\newcommand{\CH}{{\rm CH}}

\newcommand{\Pic}{{\rm Pic}}
\newcommand{\End}{{\rm End}}
\newcommand{\Hom}{{\rm Hom}}

\newcommand{\Spec}{{\rm Spec\,}}

\newcommand{\Char}{{\rm char}}
\newcommand{\Tr}{{\text{Tr}}}
\newcommand{\tr}{{\text{tr}}}

\newcommand{\0}{\emptyset}

\newcommand{\sE}{{\mathcal E}}

\newcommand{\sI}{{\mathcal I}}

\newcommand{\sK}{{\mathcal K}}

\newcommand{\sO}{{\mathcal O}}

\newcommand{\sT}{{\mathcal T}}

\newcommand{\sW}{{\mathcal W}}

\newcommand{\A}{{\mathbb A}}

\newcommand{\C}{{\mathbb C}}
\newcommand{\F}{{\mathbb F}}
\newcommand{\G}{{\mathbb G}}

\renewcommand{\P}{{\mathbb P}}
\newcommand{\Q}{{\mathbb Q}}
\newcommand{\R}{{\mathbb R}}

\newcommand{\Z}{{\mathbb Z}}

\newcommand{\EM}{{\operatorname{EM}}}
\newcommand{\BM}{{\operatorname{B-M}}}

\renewcommand{\det}{\operatorname{det}}

\newcommand{\id}{{\operatorname{\rm Id}}}

\newcommand{\Sch}{{\operatorname{\mathbf{Sch}}}} 
\newcommand{\colim}{{\mathop{\rm colim}}}

\newcommand{\op}{{\text{\rm op}}}
\newcommand{\<}{\langle}
\renewcommand{\>}{\rangle}

\newcommand{\del}{\partial}

\newcommand{\Sm}{{\mathbf{Sm}}}

\newcommand{\Proj}{{\operatorname{Proj}}}

\newcommand{\Sym}{{\operatorname{Sym}}}

\newcommand{\Gr}{{\operatorname{\rm Gr}}} 
\newcommand{\rnk}{{\operatorname{\text{rnk}}}}

\newcommand{\MGL}{{\operatorname{MGL}}}

\newcommand{\GW}{{\operatorname{GW}}} 
\newcommand{\sGW}{{\mathcal{GW}}} 
\newcommand{\SH}{{\operatorname{SH}}} 
\newcommand{\Th}{{\operatorname{Th}}} 
 
\renewcommand{\th}{{\operatorname{th}}} 
\newcommand{\sHom}{\mathcal{H}om}

\newcommand{\BGL}{\operatorname{BGL}}
\newcommand{\GL}{\operatorname{GL}}
\newcommand{\SL}{\operatorname{SL}}
\newcommand{\Sp}{\operatorname{Sp}}
\newcommand{\BSL}{\operatorname{BSL}}

\newcommand{\ev}{\text{\it ev}}

\newcommand{\pr}{\text{pr}}

\renewcommand{\Re}{\operatorname{Re}}
\renewcommand{\Im}{\operatorname{Im}}

\newcommand{\ind}[1]{}

\newcommand{\inp}[1]{}

\newcommand{\mS}{\mathbb{S}}

\renewcommand{\setminus}{-}

\begin{document}
\setcounter{tocdepth}{1}

\title[Motivic Euler characteristics]{Motivic Euler characteristics and Witt-valued characteristic classes
}

\author{Marc Levine}
\address{Fakult\"at Mathematik\\
Universit\"at Duisburg-Essen\\
Thea-Leymann-Str. 9\\
45127 Essen\\
Germany}
\email{marc.levine@uni-due.de}



\thanks{The author would like to thank the DFG for its support through the SFB Transregio 45 and the SPP 1786 ``Homotopy theory and algebraic geometry''}

\begin{abstract}  This paper examines   Euler characteristics and characteristic classes in the motivic setting. 
 We establish a motivic version of the Becker-Gottlieb transfer, generalizing a construction of Hoyois. Making calculations of the Euler characteristic of the scheme of maximal tori in a reductive group, we prove a generalized splitting principle for the reduction from $\GL_n$ or $\SL_n$ to the normalizer of a maximal torus (in characteristic zero).  Ananyevskiy's splitting principle reduces questions about characteristic classes of vector bundles in $\SL$-oriented, $\eta$-invertible theories to the case of rank two bundles. We refine the torus-normalizer splitting principle for $\SL_2$ to help compute the characteristic classes in Witt cohomology of symmetric powers of a rank two bundle, and then generalize this to develop a general calculus of characteristic classes with values in Witt cohomology. \end{abstract}
\maketitle


\section*{Introduction} A number of cohomology theories refining $K$-theory and motivic cohomology have been introduced and studied over the past 20 years. Well-known example include hermitian $K$-theory \cite{Karoubi, Schlichting}, the  theory of the Chow-Witt groups \cite{AsokFasel, BargeMorel, FaselCW, FaselSri}, Milnor-Witt motivic cohomology \cite{BachmannFasel, CalmesFasel, DegliseFasel1, DegliseFasel2, Druzhin}, Witt theory and the cohomology of the Witt sheaves  \cite{Anan}. These theories all have the property of being {\em SL-oriented}, that is, there are Thom classes for vector bundles endowed with a trivialization of the determinant line bundle. Witt theory and the cohomology of the Witt sheaves have the additional property of being {\em $\eta$-invertible}, that is multiplication by the algebraic Hopf map $\eta$ is an isomorphism on the theory. 

The $\SL$-oriented, $\eta$-invertible theories $\sE$ have been studied by Anan\-yev\-skiy \cite{Anan}. He computes the $\sE$-cohomology of $\BSL_n$ and proves a fundamental $\SL_2$-splitting principle, replacing the classical $\G_m$-splitting principle which plays a central role in the calculus of characteristic classes in $\GL$-oriented theories such as $K$-theory, motivic cohomology or $\MGL$-theory.

Works by Hoyois \cite{HoyoisGL} and Kass-Wickelgren \cite{KW1, KW2}, and well as our papers  \cite{LevVirt, LevEnum}, have initiated a study of Euler characteristics and Euler classes with values in the Grothendieck-Witt ring or in the Chow-Witt groups. In this paper we continue this study, looking at  Euler characteristics and characteristic classes with values in Witt cohomology and other $\SL$-oriented  $\eta$-invertible  theories.  

Using the six-functor formalism on the motivic stable homotopy category, we define a motivic version of the Becker-Gottlieb transfer. This includes the case of a Nisnevich locally trivial fiber bundle $f:E\to B$ with fiber $F$, such that the suspension spectrum of $F$ is a strongly dualizable object in the appropriate  motivic stable homotopy category. Just as for the classical Becker-Gottlieb transfer, the motivic transfer defines a splitting to the pullback map $f^*$ on $\sE$-cohomology for a motivic  spectrum $\sE$, as long as the categorical Euler characteristic of $F$ is invertible in the coefficient ring of $\sE$. This is discussed in 
section~\ref{sec:Transfer}. 

In section~\ref{sec:conj}, we consider a conjecture of Morel that for $N$ the normalizer of a maximal torus in a split reductive group $G$ over a perfect field $k$, the quotient scheme $N\backslash G$ has Euler characteristic one in the Grothendieck-Witt ring of $k$. This generalizes the fact that, for $N$  the normalizer of a maximal  torus in a reductive complex Lie group $G$, $N\backslash G$ has topological Euler characteristic  one. We verify a weak form of this conjecture for $\GL_n$ and $\SL_n$ over a characteristic zero field, as well as the strong form for $\GL_2$ and $\SL_2$ over an arbitrary perfect field.

We   recall the results of Ananyevskiy that we will be using here in section~\ref{sec:background}. In section~\ref{sec:Det}, we  reduce the computation of characteristic classes in an $\eta$-invertible, $\SL$-oriented theory (satisfying an additional technical hypothesis, see Theorem~\ref{thm:BSLDecomp} for the precise statement)  from the general case to that of bundles with trivialized determinant. 

Ananyevskiy's splitting principle for $\SL_n$-bundles reduces the computation of characteristic classes for such bundles in $\SL$-oriented, $\eta$-invertible theories $\sE$ to the case of $\SL_2$-bundles. In sections~\ref{secWittCoh} and \ref{sec:Rank2Bundles} we look more closely at $\SL_2$, refining Ananyevskiy's splitting principle  to reduce further to bundles with group the normalizer $N_{\SL_2}(T)$ of the standard torus $T$ in $\SL_2$. This makes computations much easier, as an irreducible representation of $N_{\SL_2}(T)$ has dimension either one or two, and the representation theory of $N_{\SL_2}(T)$ is formally (almost)  the same as that of real $O(2)$-bundles.  In section~\ref{sec:Rank2Bundles}  we introduce the rank two bundles $\tilde{O}(m)$ over $BN_{\SL_2}(T)$, as analogs of the classical bundles $O(m)$ on $B\G_m$, and show that the analogy gives a method for computing of the Euler class in Witt cohomology. We apply these considerations in section~\ref{sec:Descent} to carry out our computation of the Euler class and Pontryagin class of $\tilde{O}(m)$ in Theorem~\ref{thm:EulerClassOm}. This is used in section~\ref{sec:SymPower} to compute the characteristic classes in Witt cohomology of symmetric powers of a rank two bundle, see Theorem~\ref{thm:SymRnk2}. 

We  compute the characteristic classes in Witt cohomology for tensor products in section~\ref{sec:TensorPower}, which together with Ananyevskiy's splitting principle and our computation of the  characteristic classes  of symmetric powers of a rank two bundle gives a complete calculus for the characteristic classes in Witt cohomology. We conclude with an example, computing the ``quadratic count'' of lines on a smooth hypersurface of degree $2d-1$ in $\P^{d+1}$. This recovers results of Kass-Wickelgren for cubic surfaces \cite{KW2} and gives a purely algebraic treatment of some results of Okonek-Teleman \cite{OkonekTeleman}.

We would like to thank Eva Bayer, Jesse Kass and Kirsten Wickelgren for their comments and suggestions. Eva Bayer drew our attention to Serre's theorem on computing Hasse-Witt invariants of trace forms, and pointed out several other facts about quadratic forms, such as her paper with Suarez \cite{BayerSuarez}, which were all very important for computing the Euler characteristic of the bundles $\tilde{O}(m)$, especially our lemma~\ref{lem:Trace}. We also thank the referee for making a number of very helpful comments and suggestions.

\section{Becker-Gottlieb transfers}\label{sec:Transfer} We fix a noetherian base-scheme $S$ and let $\Sch/S$ denote the category of finite type separated $S$-schemes. 

We have the motivic stable homotopy category 
\[
\SH(-):\Sch/S^\op\to \mathbf{Tr}
\]
where $\mathbf{Tr}$ is the 2-category of symmetric monoidal triangulated categories. Following \cite{Ayoub, CD, HoyoisEquiv}, the functor $\SH(-)$ admits the Grothendieck six operations; we briefly review the aspects of this structure that we will be needing here. We note that the theory in  \cite{Ayoub, HoyoisEquiv} is for quasi-projective $S$-schemes; the extension to separated $S$-schemes of finite type is accomplished in  \cite[Theorem 2.4.50]{CD}..

For each $X$, $\SH(X)$ is the homotopy category of a closed symmetric stable model category  \cite{Jardine}, which makes $\SH(X)$ into a closed symmetric monoidal category. We denote the symmetric monoidal product in $\SH(X)$ by $\wedge_X$ and the adjoint internal Hom by $\sHom_X(-,-)$. The unit   $1_X$ is the  motivic sphere spectrum over $X$, $\mS_X$.  We have for each morphism $f:X\to Y$ in $\Sch/S$ the exact symmetric monoidal functor $f^*:\SH(Y)\to \SH(X)$ with right adjoint $f_*:\SH(X)\to \SH(Y)$, and exact functors $f^!:\SH(Y)\to \SH(X)$, $f_!:\SH(X)\to\SH(Y)$, with $f_!$ left adjoint to $f^!$. 

We will refer to a  commutative monoid $\sE$ in the symmetric monoidal category $\SH(X)$ as a {\em motivic ring spectrum}, that is, $\sE$ is endowed with a multiplication map $\mu:\sE\wedge_X\sE\to \sE$ and a unit map $e:1_X\to \sE$ such that $(\sE, \mu, e)$ is a commutative monoid in $(\SH(X), \wedge_X, 1_X)$.

The adjoint pair $f_!\dashv f^!$ satisfies a projection formula and the adjunctions extend to internal adjunctions: for $f:X\to Y$, there are natural isomorphisms
\begin{align*}
&f_!(f^*y\wedge_X x)\cong y\wedge_Yf_!x\\
&f_*\sHom_X(x, f^!x)\cong \sHom_Y(f_!x, y)
\end{align*}
There is a natural transformation $\eta^f_{!*}:f_!\to f_*$, which is an isomorphism if $f$ is proper (see e.g. \cite[\S 2.2.6]{CD}). Furthermore, for $f$ smooth, the functor $f^*$ admits a left adjoint $f_\#$.  

If $i:X\to Y$ is a closed immersion with open complement $j:U\to Y$, there are localization distinguished triangles
\[
j_!j^!\to \id_{\SH(Y)}\to i_*i^*
\]
and
\[
i_!i^!\to  \id_{\SH(Y)}\to j_*j^*.
\]
In addition, $j_!=j_\#$, $j^!=j^*$ and $i_!=i_*$, the counit $i^*i_*\to \id$ is an isomorphism and the pair $(i^*, j^*)$ is conservative. This gives the following {\em weak Nisnevich descent property}: Let $\{g_i:U_i\to Y\}$ be a Nisnevich cover.  A morphism $f:a\to b$ in $\SH(Y)$ is an isomorphism if and only if $g_i^*f:g_i^*a\to g_i^*b$ is an isomorphism in $\SH(U_i)$ for all $i$ (see \cite[Proposition 3.2.8, Corollary 3.3.5]{CD}).

If $p:V\to X$ is a vector bundle with 0-section $s:X\to V$, we have the endomorphism $\Sigma^V:\SH(X)\to \SH(X)$ defined as $\Sigma^V:=p_\#s_*$. This is an auto-equivalence, with quasi-inverse $\Sigma^{-V}:=s^!p^*$.  For $f:X\to Y$ smooth, there are canonical isomorphisms
\[
f_!\cong f_\#\circ\Sigma^{-T_f};\quad f^!\cong \Sigma^{T_f}\circ f^*
\]
where $T_f$ is the relative tangent bundle of $f$. 

For $\pi_Y:Y\to X$ in $\Sm/X$, we let $Y/X\in \SH(X)$ denote the object $\pi_{Y\#}(1_Y)$ and for $\pi_Y:Y\to X$ in $\Sch/X$, we let $Y_\BM/X\in \SH(X)$ denote the object $\pi_{Y!}(1_Y)$. Sending $Y\in \Sm/X$ to $Y/X$ defines a functor 
\[
-/X:\Sm/X\to \SH(X)
\]
and sending $Y\in \Sch/X$ to $Y_\BM/X$ defines a functor
\[
-_\BM/X:\Sch^\pr/X^\op\to \SH(X),
\]
where $\Sch^\pr/X$ is the subcategory of proper morphisms in $\Sch/X$. 
For a morphism $g:Y\to Z$ in $\Sm/X$, the map $g/X:Y/X\to Z/X$ is induced by the co-unit of the adjunction $e_g:g_!g^!\to \id_{\SH(Z)}$ as the composition
\begin{multline*}
\pi_{Y\#}(1_Y)=\pi_{Y\#}\pi_Y^*(1_X)\cong \pi_{Y!}\pi_Y^!(1_X)\cong 
\pi_{Z!}g_!g^!\pi_Z^!(1_X)\\\xrightarrow{e_g}
\pi_{Z!}\pi_Z^!(1_X)\cong \pi_{Z\#}\pi_Z^*(1_X)=\pi_{Z\#}(1_Z).
\end{multline*}
For a proper morphism $g:Z\to Y$ in $\Sch/X$ the map $g/X_\BM:Y/X_\BM\to Z/X_\BM$ is  the composition
\[
\pi_{Y!}(1_Y)\xrightarrow{u_g}\pi_{Y!}g_*g^*(1_Y)\xrightarrow{(\eta^g_{!*})^{-1}}
\pi_{Y!}g_!g^*(1_Y)\cong \pi_{Z!}(1_Z),
\]
where $u_g:\id_{\SH(Y)}\to g_*g^*$ is the unit of the adjunction. 

\begin{remark} Using the construction of $\SH(X)$ as the homotopy category of the category of symmetric $T$-spectra with respect to the $T$-stable $\A^1$-Nisnevich model structure \cite{Jardine}, the object $Y/X$ is canonically isomorphic to the usual $T$-suspension spectrum $\Sigma^\infty_TY_+$ and for $g:Y\to Z$ a morphism in $\Sm/X$, the map $g/X:Y/X\to Z/X$ is the map $\Sigma_T^\infty g_+:\Sigma^\infty_TY_+\to \Sigma^\infty_TZ_+$. We mention this only to orient the reader; we will not be needing this fact in our discussion.
\end{remark}

For $\pi_Y:Y\to X$ in $\Sm/X$, there is a morphism
\[
\ev_Y:Y_\BM/X\wedge_X Y/X\to 1_X
\]
defined as follows: We consider $Y\times_XY$ as a  $Y$-scheme via $p_2$. We have the canonical isomorphisms
\begin{multline*}
\pi_{Y\#}(p_{2!}(1_{Y\times_XY}))\cong \pi_{Y\times_XY\#}(\Sigma^{-T_{p_2}}(1_{Y\times_XY}))\\\cong \pi_{Y\#}(\Sigma^{-T_{Y/X}}(1_Y))\wedge_X\pi_{Y\#}(1_Y)\\\cong
\pi_{Y!}(1_Y)\wedge_X\pi_{Y\#}(1_Y)\cong
Y_\BM/X\wedge_X Y/X.
\end{multline*}
The diagonal section $s_\delta:Y\to Y\times_XY$ to $p_2$ is a proper morphism in $\Sch/Y$, giving the map
\[
s_\delta^*:p_{2!}(1_{Y\times_XY})\to 1_Y
\]
in $\SH(Y)$. Applying $\pi_{Y\#}$ and using the above-mentioned isomorphism gives the map
\[
\pi_{Y\#}(s_\delta^*):Y_\BM/X\wedge_X Y/X\to Y/X.
\]
The map $\ev_Y$ is then defined as the composition
\[
Y_\BM/X\wedge_X Y/X\xrightarrow{\pi_{Y\#}(s_\delta^*)} Y/X\xrightarrow{\pi_Y/X} X/X=1_X.
\]

For $\pi_Y:Y\to X$ a proper morphism in $\Sch/X$, we have the map
\[
\pi_{Y\BM}/X:1_X\to Y_\BM/X.
\]
If $Y$ is smooth over $X$, we consider $Y\times_XY$ as a smooth $Y$-scheme via $p_1$, giving the map
\[
s_\delta/Y:1_Y=Y/Y\to Y\times_XY/Y=p_{1\#}(1_{Y\times_XY}).
\]
Applying $\pi_{Y!}$ and using the canonical isomorphism
\[
\pi_{Y!}(p_{1\#}(1_{Y\times_XY}))\cong  Y/X\wedge_XY_\BM/X
\]
constructed as above gives the map
\[
\pi_{Y!}(s_\delta/Y):Y_\BM/X\to Y/X\wedge_XY_\BM/X.
\]
For $Y$ smooth and proper over $X$, define the map $\delta_Y:1_X\to Y/X\wedge_XY_\BM/X$ as the composition
\[
1_X\xrightarrow{\pi_{Y\BM}/X}Y_\BM/X\xrightarrow{\pi_{Y!}(s_\delta/Y)}Y/X\wedge_XY_\BM/X.
\]

Recall that an object $x\in \SH(X)$ has the {\em dual} $x^\vee:=\sHom_X(x, 1_X)$, giving the evaluation map
\[
\ev: x^\vee\wedge_X x\to 1_X
\]
adjoint to the identity on $x^\vee$. The map $\ev$ in turn
induces maps   $can_a:x^\vee\wedge_Xa\to \sHom_X(x, a)$ and $can'_a: a\wedge_Xx\to \sHom_X(x^\vee, a)$;  $x$ is {\em strongly dualizable} if the morphism
$can_x:x^\vee\wedge_Xx\to \sHom_X(x, x)$ is an isomorphism.   This is equivalent to the existence of an object $y\in \SH(X)$ and morphisms $\delta_x:1_X\to x\wedge_Xy$, $\ev_x:y\wedge_X x\to 1_X$, such that 
\begin{align}\label{align:DualityIds}
&(\id_x\wedge_X\ev_x)\circ(\delta_x\wedge_X\id_x)=\id_x;\\
&(\ev_x\wedge_X\id_y)\circ(\id_y\wedge_X\delta_x)=\id_y.\notag
\end{align}
(see \cite[Theorem 2.6]{MayPic}). In this case, there is a canonical isomorphism $y\cong x^\vee$ with $\ev_x$ going over to $\ev$, and $\delta_x$ going over to the map
\[
1_X\to x\wedge x^\vee
\]
given as the image of $\id_x$ under the sequence of isomorphisms
\begin{align*}
\Hom_{\SH(X)}(x,x)&\cong 
\Hom_{\SH(X)}(1_X, \sHom_X(x,x))\\
&\xleftarrow{can_{x*}} \Hom_{\SH(X)}(1_X,  x^\vee\wedge_X x)\\
&\xrightarrow{\tau_{x^\vee, x*}}  \Hom_{\SH(X)}(1_X,  x\wedge_X x^\vee).
\end{align*}

\begin{proposition}\label{prop:DualProp} 1. For $f:Y\to X$ smooth and proper, the triple $(Y_\BM/X, \delta_Y, \ev_Y)$ exhibits $Y/X$ as a strongly dualizable object of $\SH(X)$ with dual $Y_\BM/X$.\\
2. Let $\{j_i:U_i\to X, i=1,\ldots, n\}$ be a finite Nisnevich open cover of $X$. Suppose that for $x\in \SH(X)$, $j_i^*x\in \SH(U_i)$ is strongly dualizable for all $i$. Then $x$ is a strongly dualizable object of $\SH(X)$.\\
3. The class of strongly dualizable objects in $\SH(X)$ is closed under $\wedge_X$ and
\[
(x_1\wedge_Xx_2)^\vee\cong x_1^\vee\wedge_Xx_2^\vee.
\]
4. If $f:X\to Y$ is a morphism in $\Sch/S$ and $y\in \SH(Y)$ is strongly dualizable, then $f^*y$ is strongly dualizable with dual $f^*(y^\vee)$.\\
5. The class of strongly dualizable objects in $\SH(X)$ satisfies the 2 out of 3 property with respect to distinguished triangles: if $x_1\xrightarrow{u} x_2\xrightarrow{v} x_3\xrightarrow{w} x_1[1]$ is a distinguished triangle in $\SH(X)$ and two of the $x_i$ are strongly dualizable, then so is the third. Moreover, the triangle
\[
x_3^\vee\xrightarrow{v^\vee}x_2^\vee\xrightarrow{u^\vee} x_1^\vee\xrightarrow{-w^\vee[1]}x_3[1]
\]
is distinguished.
\end{proposition}

\begin{proof} (1) follows from \cite[Theorem 2.4.50, Proposition 2.4.31]{CD}. For (2), take $a\in \SH(X)$ and consider the map $can_x:x^\vee\wedge_X x\to \sHom_X(x,x)$  Then $j_i^*can$ is the canonical map $:j_i^*x^\vee\wedge_{U_i} j_i^*x\to \sHom_{U_i}(j_i^*x,j_i^*x)$, after making the  evident identifications. Thus $j_i^*can$ is an isomorphism for all $i$ and $can$ is then an isomorphism  by the weak Nisnevich descent property.  (3) follows from the characterization of strongly dualizable objects via the existence of the maps $\delta,\ev$ satisfying the identities \eqref{align:DualityIds}; (4) is similar, noting that $f^*$ is a symmetric monoidal functor. The assertion (5) is proven by May \cite[Theorem 0.1]{May} for a symmetric monoidal triangulated category subject to certain axioms, which are satisfied by $\SH(X)$.
\end{proof}

\begin{definition} Let $x\in \SH(X)$ be a strongly dualizable object. The {\em Euler characteristic} $\chi(x/X)\in \End_{\SH(X)}(1_X)$ is the composition
\[
1_X\xrightarrow{\delta_x}x\wedge_Xx^\vee\xrightarrow{\tau_{x, x^\vee}}x^\vee\wedge_Xx\xrightarrow{\ev_x}1_X.
\]
\end{definition}

\begin{remark} If $X=\Spec F$ for $F$ a perfect field, then Morel's theorem   \cite[Theorem 6.4.1]{MorelICTP}  identifies $\End_{\SH(X)}(1_X)$ with the Grothendieck-Witt group $\GW(F)$. 
\end{remark}

\begin{definition} \label{def:transfer} 1. Let $f:E\to B$ be in $\Sm/B$ with $E/B$ a strongly dualizable object in $\SH(B)$. The {\em relative transfer} $\Tr_{f/B}:1_B\to E/B$ is defined as follows: Consider the diagonal morphism 
\[
\Delta_E:E\to E\times_BE.
\]
Applying $f_\#$ gives the morphism
\[
\Delta_{E}/B:E/B\to E\times_BE/B\cong E/B\wedge_BE/B.
\]
By duality, we have the isomorphism
\[
(-)^\tr:\Hom_{\SH(B)}(E/B, E/B\wedge_BE/B)\cong \Hom_{\SH(B)}(E/B^\vee\wedge_BE/B, E/B),
\]
defined by sending a morphism $g:E/B\to E/B\wedge_BE/B$ to the composition
\[
E/B^\vee\wedge_BE/B\xrightarrow{\id\wedge g}E/B^\vee\wedge_BE/B\wedge_BE/B\xrightarrow{\ev_{E/B}\wedge\id}1_B\wedge_BE/B\cong E/B.
\]
We set
\[
\Tr_{f/B}:=\Delta_{E}/B^\tr\circ\tau_{E/B,E/B^\vee}\circ \delta_{E/B}.
\]
2. Suppose we have $f:E\to B$ in $\Sm/B$ as in (1), with $\pi_B:B\to S$ in $\Sm/S$, giving the structure morphism $\pi_E:E\to S$, $\pi_E:=\pi_B\circ f$. Define the {\em transfer} $\Tr(f/S):B/S\to E/S$ by
\[
\Tr(f/S):=\pi_{B\#}(\Tr_{f/B}),
\]
where we use the canonical isomorphism
\[
E/S:=\pi_{E\#}(1_E)\cong \pi_{B\#}(f_\#(1_E))=\pi_{B\#}(E/B).
\]
\end{definition}

\begin{lemma} Given $f:E\to B$ as in Definition~\ref{def:transfer}(1), let $g:B'\to B$ be a morphism in $\Sch/S$ and let $f':E'\to B'$ be the pull-back $E':=E\times_BB'$, $f'=p_2$. Then $f':E'\to B'$ is strongly dualizable in $\SH(B')$, there are canonical isomorphisms $E'/B'\cong g^*(E/B)$, $1_{B'}\cong g^*(1_B)$ and via these isomorphisms, we have
\[
\Tr(f'/B')=g^*(\Tr(f/B)).
\]
Moreover, if $\pi_B:B\to S$ is in $\Sm/S$, if $g_0:S'\to S$ is a morphism in $\Sch/S$, if $B'=S'\times_SB$ and $g:B'\to B$ is the projection, then there is a canonical isomorphism  $B'/S'\cong g_0^*(B/S)$ and via this isomorphism
\[
\Tr(f'/S')=g_0^*(\Tr(f/S)).
\]
\end{lemma}

\begin{proof} For the the first assertion, we have the exchange isomorphism (see \cite[\S1.1.6]{CD} or \cite[\S1.4.5]{Ayoub})
\[
Ex^*_\#:p_{1\#}\circ p_2^*\to g^*\circ f_\#\
\]
associated to the Cartesian diagram
\[
\xymatrix{
E\times_BB'\ar[r]^{p_2}\ar[d]_{p_1}&B'\ar[d]^g\\
E\ar[r]_f&B,
}
\]
which gives the isomorphism
\[
g^*(E/B)=g^*\circ f_\#(1_E)\xrightarrow{(Ex^*_\#)^{-1}}p_{1\#}\circ p_2^*(1_E)\cong p_{1\#}(1_{E'})=E'/B'.
\]
We have as well the canonical isomorphism $1_{B'}\cong g^*(1_B)$. Since $g^*$ is a symmetric monoidal functor, the assumption that $E/B$ is strongly dualizable in $\SH(B)$ implies that $E'/B'$ is strongly dualizable in $\SH(B')$ with ${E'/B'}^\vee$ canonically isomorphic to $g^*(E/B^\vee)$. Moreover, via these isomorphisms, we have
\[
\delta_{E'/B'}=g^*(\delta_{E/B}), \ \ev_{E'/B'}=g^*(\ev_{E/B}).
\]
Finally, since the exchange isomorphism is a natural transformation, we have
\[
g^*({\Delta_E/B}^\tr)={\Delta_{E'}/B'}^\tr
\]
and thus 
\begin{align*}
\Tr(f'/B)&:={\Delta_{E'}/B'}^\tr\circ\tau_{E'/B', {E'/B'}^\vee}\circ\delta_{E'/B'}\\
&=g^*({\Delta_E/B}^\tr)\circ g^*(\tau_{E/B, {E/B}^\vee})\circ g^*(\delta_{E/B})\\
&=g^*(\Tr(f/B)).
\end{align*}

The proof that $\Tr(f'/S')=g_0^*(\Tr(f/S))$ is similar and is left to the reader.
\end{proof}

\begin{lemma} Let $f:E\to B$ be in $\Sm/B$, with $E/B$ strongly dualizable in $\SH(B)$. Then 
\[
f/B\circ \Tr(f/B):1_B\to 1_B
\]
is equal to $\chi(E/B)$. 
\end{lemma}

\begin{proof} $f/B\circ \Tr(f/B)$ is the composition
\[
1_B\xrightarrow{\delta_{E/B}}E/B\wedge_BE/B^\vee\xrightarrow{\tau_{E/B,E/B^\vee}}
E/B^\vee\wedge_BE/B
\xrightarrow{\Delta_E/B^\tr}E/B\xrightarrow{f/B}1_B.
\]
The map $\Delta_E/B^\tr$ is the composition
\begin{multline*}
E/B^\vee\wedge_BE/B\xrightarrow{\id\wedge\Delta_E/B}
E/B^\vee\wedge_BE/B\wedge_BE/B\\ \xrightarrow{\ev_{E/B}\wedge\id}
1_B\wedge_BE/B\cong E/B
\end{multline*}
We have the commutative diagram
\[
\xymatrixcolsep{40pt}
\xymatrix{
&E/B^\vee\wedge_BE/B\ar[dl]_{\id\wedge\Delta_E/B}\ar@/^60pt/[ddd]^{\ev_{E/B}}\\
E/B^\vee\wedge_BE/B\wedge_BE/B \ar[d]_{\ev_{E/B}\wedge\id}
\ar[r]^{\id\wedge\id\wedge f/B}&E/B^\vee\wedge_BE/B\wedge_B1_B\ar[d]^{\ev_{E/B}\wedge\id}\ar[u]_\sim\\
1_B\wedge_BE/B\ar[d]^\wr&1_B\wedge_B1_B\ar[d]^\wr\\
E/B\ar[r]_{f/B}&1_B
}
\]
Thus
\[
f/B\circ \Tr(f/B)=\ev_{E/B}\circ\tau_{E/B,E/B^\vee}\circ \delta_{E/B}=\chi(E/B).
\]
\end{proof}

\begin{remark} Hoyois \cite[Remark 3.5]{HoyoisGL} has defined the relative transfer for a smooth proper map $f:E\to B$ using the six-functor formalism, and has noted the identity $f/B\circ \Tr(f/B)=\chi(E/B)$ in this case.
\end{remark}

\begin{lemma} Suppose $\pi_B:B\to S$ is in $\Sm/S$, and  $E=F\times_SB$ for some  $\pi_F:F\to S$ in $\Sm/S$, with $f:E\to B$ the projection. Suppose that $F/S$ is strongly dualizable. Then $E/B$ is strongly dualizable and $f/S\circ\Tr(f/S)$ is multiplication by $\chi(F/S)$.
\end{lemma}

\begin{proof} Under our assumptions, we have $E/B=\pi_B^*(F/S)$. Since $\pi_B^*$ is a symmetric monoidal functor, it follows that $E/B^\vee=\pi_B^*(F/S^\vee)$ and $\Tr(E/B)=\pi_B^*(\Tr(F/S))$. Thus 
\begin{align*}
f/S\circ \Tr(E/S)&=\pi_{B\#}(f/B\circ\pi_B^*(\Tr(F/S))),\\
&=\pi_{B\#}(\pi_B^*(\pi_F/S)\circ \pi_B^*(\Tr(F/S)))\\
&=\pi_{B\#}(\pi_B^*(\pi_F/S\circ \Tr(F/S)))\\
&=\pi_{B\#}(\pi_B^*(\chi(F/S)))
\end{align*}
which is multiplication by $\chi(F/S)$.
\end{proof}

Let $f:E\to B$ be a finite type morphism, with $B\in \Sch/S$ and let $F$ be in $\Sch/S$. We say that $f$ is a Nisnevich locally trivial bundle  with fiber $F$ if there is a covering family $\{g_i:U_i\to B, i\in I\}$ for the Nisnevich topology on $B$ and isomorphisms of $U_i$-schemes $E\times_BU_i\cong  F\times_SU_i$ for each $i\in I$.

\begin{theorem} \label{thm:motBG} Let $F$ and $B$ be in $\Sm/S$ and let $f:E\to B$ in $\Sm/B$ be a  Nisnevich locally trivial bundle  with fiber $F$. \\
1. If $F/S$ is strongly dualizable in $\SH(S)$, then $E/B$ is strongly dualizable in $\SH(B)$.\\
2. Suppose that $F/S$ is strongly dualizable in $\SH(S)$. Let $\sE\in\SH(S)$ be a motivic ring spectrum  such that the map $\chi(F/S):1_S\to 1_S$ induces an isomorphism on $\sE^{0,0}(S)$. Then the map
\[
f/S^*:\sE^{**}(B/S)\to \sE^{**}(E/S)
\]
is split injective. 
\end{theorem}

\begin{proof} The assertion (1)  follows from the isomorphism $g_i^*(E/B)\cong F\times_SU_i/U_i$, and Proposition~\ref{prop:DualProp}(2, 4). 
 
 For (2), we have 
 \begin{multline*}
 \sE^{a,b}(E/S)=\Hom_{\SH(S)}(\pi_{B\#}(E/B), \Sigma^{a,b}\sE)\\=
 \Hom_{\SH(B)}(E/B, \Sigma^{a,b}\pi_B^*\sE)=\pi_B^*\sE^{a,b}(E/B);
 \end{multline*}
similarly $\sE^{a,b}(B/S)=\pi_B^*\sE^{a,b}(1_B)$ and via these isomorphisms, the map $f/S^*:\sE^{a,b}(B/S)\to \sE^{a,b}(E/S)$ is the map $f/B^*: \pi_B^*\sE^{a,b}(1_B)\to \pi_B^*\sE^{a,b}(E/B)$.
We claim that the endomorphism
 \[
 \Tr(E/B)^*\circ f/B^*:\pi_B^*\sE^{a,b}(1_B)\to \pi_B^*\sE^{a,b}(1_B)
 \]
 is an isomorphism; assuming this the case, the map $ \Tr(E/B)^*$ gives a left quasi-inverse to $f/S^*$.  
 
 By weak Nisnevich descent, it suffices to show that the pullback
 \[
 g_i^*( \Tr(E/B)^*\circ f/B^*):g_i^*\pi_B^*\sE^{a,b}(1_{U_i})\to g_i^*\pi_B^*\sE^{a,b}(1_{U_i})
 \]
 is an isomorphism for all $i$. But 
 \[
 g_i^*( \Tr(E/B)^*\circ f/B^*)= \Tr(F\times_SU_i/U_i)^*\circ p_{U_i}^*
 \]
 and $p_{U_i}\circ \Tr(F\times_SU_i/U_i)$ is multiplication by $\pi_{U_i}^*\chi(F/S)$. Using the multiplication $\mu$ in $\sE$, we have the commutative diagram
 \[
 \xymatrixcolsep{60pt}
 \xymatrix{
 g_i^*\pi_B^*\sE^{a,b}(1_{U_i})\otimes \sE^{0,0}(1_{U_i})\ar[d]_\mu\ar[r]^{\id\otimes \times\pi_{U_i}^*\chi(F/S)}&g_i^*\pi_B^*\sE^{a,b}(1_{U_i})\otimes \sE^{0,0}(1_{U_i})\ar[d]^\mu\\
 g_i^*\pi_B^*\sE^{a,b}(1_{U_i})\ar[r]^{g_i^*(\Tr(E/B)^*\circ f/B^*)}&
 g_i^*\pi_B^*\sE^{a,b}(1_{U_i}) 
 }
 \]
 and so  $g_i^*( \Tr(E/B)^*\circ f/B^*)$ is an isomorphism by our assumption on $\sE$. 
\end{proof}

\begin{remark} The above result is a motivic analog of the theorem of Becker-Gottlieb \cite[Theorem 5.7]{BG}.
\end{remark}

\section{The case of split reductive groups: a conjecture}\label{sec:conj}
Let $G$ be a split reductive group-scheme over a base-scheme $S$, with Borel subgroup $B$ and maximal split torus $T\subset B$. Let $N_T\supset T$ be the normalizer of $T$ in $G$. We take the model $N_T\backslash EG$ for $BN_T$ (see the end of this section for details on our model for $EG$) and  consider the $N_T\backslash G$-bundle $BN_T\to BG$. For {\em special} $G$, that is, for $G=\GL_n, \SL_n, \Sp_{2n}$, $BN_T\to BG$ is a Nisnevich locally trivial bundle, and so we can apply the motivic Becker-Gottlieb theorem developed in the previous section; for this, we need to compute the Euler characteristic $\chi(N_T\backslash G/S)$.

Following a suggestion of Fabien Morel, who formulated the strong form, we have the following conjecture.

\begin{conjecture}\label{conj:RedConj} {\bf Strong form}: For $G$ a split reductive group-scheme over a perfect field $k$, we have $\chi((N_T\backslash G)/k)=1$ in $\GW(k)$.\\[5pt]
{\bf Weak form}: For $G$ a split reductive group-scheme over a perfect field $k$, the Euler characteristic $\chi((N_T\backslash G)/k)$ is invertible in $\GW(k)$.
\end{conjecture}

The weak form suffices to apply the motivic Becker-Gottlieb transfer to the map $\pi_G:BN_T(G)\to BG$ in case the $N_T(G)\backslash G$-bundle is Nisnevich locally trivial; see Theorem~\ref{thm:BGSplitting} below for a precise statement in case $G=\GL_n, \SL_n$.

For $S=\Spec\C$ it is well-known that $N_T(\C)\backslash G(\C)$ has (topological) Euler characteristic 1. This is also easy to prove: $T(\C)\backslash  G(\C)$ is homotopy equivalent to $B(\C)\backslash  G(\C)$, which has homology the free abelian group on the closure of the Schubert cells $B(\C)wB(\C)$, $w\in W(G):=N_T(\C)/T(\C)$; as these are all of even dimension, we have $\chi^{top}(G(\C)/T(\C))=\#W(G)$. As the 
covering space $T(\C)\backslash  G(\C)\to N_T(\C)\backslash  G(\C)$ has degree $\#W(G)$, it follows that $\chi^{top}(N_T(\C)\backslash  G(\C))=1$. 

We note that $N_T(G)(\C)\backslash  G(\C)=(N_T(G)\backslash  G)(\C)$, since $\C$ is algebraically closed. For $\R$, this is no longer the case, which makes the analysis of the Euler characteristic $\chi^{top}((N_T(G)\backslash  G)(\R))$ a bit more difficult.

\begin{lemma}\label{lem:EulerCharRComp} For $G=\GL_n$ or $G=\SL_n$, we have $\chi^{top}((N_T(G)\backslash  G)(\R))=1$.
\end{lemma}

\begin{proof} It suffices to   consider the case $G=\GL_n$, since   the schemes $N_T(\GL_n)\backslash \GL_n$ and $N_T(\SL_n)\backslash \SL_n$ are isomorphic. We write $N_T$ for $N_T(\GL_n)$.

$N_T=\G_m^n\rtimes \mathfrak{S}_n$, where $\mathfrak{S}_n$ is the symmetric group,  so we have
\[
N_T\backslash \GL_n=\mathfrak{S}_n\backslash(\G_m^n\backslash \GL_n)
\]
and $\G_m^n\backslash \GL_n$ is the open subscheme of $(\P^{n-1})^n$ parametrizing $n$-tuples of lines through 0 in $\A^n$ which span $\A^n$. This realizes 
$N_T\backslash \GL_n$ as an open subscheme of $\Sym^n\P^{n-1}$, and gives a decomposition of the $\R$-points $x\in (N_T\backslash \GL_n)(\R)$ in terms of the residue fields of the corresponding closed points $y_1,\ldots, y_s$ of $(\P^{n-1})^n$ lying over $x$. This decomposes $(N_T\backslash \GL_n)(\R)$ as a disjoint union of open submanifolds
\[
(N_T\backslash \GL_n)(\R)=\amalg_{i=0}^nU_i
\]
where $U_i$ parametrizes the collection of lines with exactly $i$ real lines. Of course $U_i$ is empty if $n-i$ is odd. 

We claim that $\chi^{top}(U_i)=0$ if $i\ge2$. Indeed, we can form a finite covering space $\tilde{U}_i\to U_i$ by killing the permutation action on the real lines and choosing an orientation on each real line. Then putting a metric on $\C^n$, we see that $\tilde{U}_i$ is homotopy equivalent to an $SO(i)$-bundle over 
the space parametrizing the (unordered) collection of the remaining lines. Since $\chi^{top}(SO(i))=0$ for $i\ge2$, we have $\chi^{top}(\tilde{U}_i)=0$ and hence $\chi^{top}(U_i)=0$ as well.

Now assume $n=2m$ is even. $U_0$ parametrizes the unordered collections of $m$ pairs of complex conjugate lines that span $\C^n$, so we may form the degree $m!$ covering space $\tilde{U}_0\to U_0$ of ordered $m$-tuples of such conjugate pairs. Each pair $\ell, \bar{\ell}$ spans a $\A^2\subset \A^n$ defined over $\R$, so we may map $\tilde{U}_0$ to the real points of the  quotient $(\GL_2)^m\backslash \GL_n$, which is homotopy equivalent to the real points of the partial flag manifold $Fl_m:=P_m\backslash \GL_n$, where $P_m$ is the parabolic with Levi subgroup 
$(\GL_2(\C))^m$. In fact, the map $\tilde{U}_0\to Fl_m(\R)$ is also a homotopy equivalence: given a  two-plane $\pi\subset \A^n$ defined over $\R$, the pairs of conjugate lines in $\pi$ are parametrized by the upper half-plane in the Riemann sphere $\P(\pi)$, which is contractible.

We claim that $\chi^{top}(Fl_m(\R))=m!$ for all $m\ge2$. To compute the Euler characteristic, we use the following fact:\\[10pt]
If $X\in \Sm/\R$ is {\em cellular}, that is, $X$ admits a finite stratification by locally closed subschemes $X_i\cong \A^{n_i}$, then $\chi^{top}(X(\R))=\rnk\,\CH^{even}(X)-
\rnk\,\CH^{odd}(X)$.
\\[10pt]
This follows from \cite[Remark 1.11, Proposition 1.14]{LevEnum}, using the $\GW(\R)$-valued Euler characteristic $\chi(X/\R)$:
\begin{equation}\label{eqn:EulerCharCell}
\vbox{
\ \\
\noindent
i. For $X\in \Sm/\R$,  $\chi^{top}(X(\R))$ is equal to the signature of $\chi(X/\R)$  \cite[Remark 1.11]{LevEnum}. \\[1pt]
ii. For $X\in \Sm/k$ cellular, $\chi(X/k)=\rnk\,\CH^{even}\cdot\<1\>+\rnk\,\CH^{odd}\cdot\<-1\>$ \cite[Proposition 1.14]{LevEnum}.
}
\end{equation}
 Alternatively, one can apply the Lefschetz trace formula to complex conjugation  $c$ acting on $H^*(X(\C), \Q)$, and use the fact that for $X$ cellular, $H^{odd}(X(\C),\Q)=0$, $H^{2n}(X(\C),\Q)$ has dimension $\rnk\,\CH^n(X)$ and $c$ acts by $(-1)^n$, since the cycle class map $\cl^n$ is a $c$-equivariant isomorphism from the $c$-invariant $\Q$ vector space $\CH^n(X)\otimes\Q$ to $H^{2n}(X(\C),\Q(2\pi i)^n)=\Q\cdot \cl^n(\CH^n(X))$.

For $m=2$, $Fl_2=\Gr(2,4)$, which has 4 Schubert cells of even dimension and 2 of odd dimension. Thus $\chi(Fl_2/\R)=4\<1\>+2\<-1\>$, which has signature 2, and hence $\chi^{top}(Fl_2(\R))=2$. 

In general, we have the fibration $Fl_m\to Fl_{m-1}$ with fiber $\Gr(2, 2m)$. We can compare the Schubert cells for $\Gr(2, 2m)$ with those for $\Gr(2,2m-2)$: the ``new'' ones come from the partitions $(2m-3,i)$, $i=0, \ldots, 2m-3$ and $(2m-2, i)$, $i=0,\ldots, 2m-2$. This  adds the quadratic form $(2m-1)\<1\>+ (2m-2)\<-1\>$ of signature +1 to  $\chi(\Gr(2,2m-2))$, which implies that $\chi^{top}(\Gr(2,2m)(\R))=\chi^{top}(\Gr(2, 2m-2))+1$. By induction, this says that $\chi^{top}(\Gr(2,2m)(\R))=m$, and similarly by induction we find $\chi^{top}(Fl_m)=m\cdot \chi^{top}(Fl_{m-1})=m!$. Thus $\chi^{top}(\tilde{U}_0)=m!$ and hence $\chi^{top}(U_0)=1$.

In case $n=2m+1$ is odd, we can again choose a metric and fiber $U_1$ over the real projective space $\R\P^{2m}$, with fiber over $[\ell]\in \R\P^{2m}$ homotopy equivalent to the $U_0$ for the hyperplane perpendicular to the $\C$-span of $\ell$. As $\chi^{top}(\R\P^{2m})=1$, this yields $\chi^{top}(U_1)=1$. 
\end{proof}

\begin{proposition} Let $n\ge1$ be an integer. There is a 4-torsion element $\tau_n\in \GW(\Q)$ with  $\chi((N_T\backslash\GL_n)/\Q)=1+\tau_n$.
Moreover, the Euler characteristic $\chi((N_T\backslash\GL_n)/\Q)$ is invertible in $\GW(\Q)$.
\end{proposition}

\begin{proof} From Lemma~\ref{lem:EulerCharRComp} and \eqref{eqn:EulerCharCell}(i), $\chi((N_T\backslash\GL_n)/\Q)$ has signature 1. Moreover,  the rank of $\chi((N_T\backslash\GL_n)/\Q)$ is $\chi^{top}((N_T\backslash\GL_n)(\C))=1$ and thus $\chi((N_T\backslash\GL_n)/\Q)-1\in \GW(\Q)$ has rank and signature zero. The Witt ring of $\Q$ is isomorphic to the direct sum $W(\R)\oplus\oplus_p W(\F_p)$ \cite[Chap. 5, Theorem 1.5 and Theorem 3.4]{Scharlau} and thus an element of $\GW(\Q)$ with rank and signature zero lives in the direct sum of the  Witt rings $W(\F_p)$, which are either $\Z/2$ (for $p=2$  \cite[Chap. 5, Theorem 1.6]{Scharlau}),   $\Z/2\times \Z/2$ or $\Z/4$ \cite[Chap. 2, Corollary 3.11]{Scharlau}. As the kernel $I$ of $\GW(\Q)\to \GW(\R)$ satisfies $I^3=0$ \cite[Chap. 5, Theorem 1.5, Theorem 3.4 and Theorem 6.6]{Scharlau}, an element of $\GW(\Q)$ with rank and signature one is invertible.
\end{proof}

\begin{corollary} The weak form of Conjecture~\ref{conj:RedConj} holds for $G=\GL_n$ and $G=\SL_n$ over a field $k$ of characteristic zero.
\end{corollary}

We take the opportunity here to make explicit our model for $BG$, $G$ a reductive group-scheme over a perfect field $k$. We follow Morel-Voevodsky \cite[\S 4.2.2]{MorelVoev}, by using the Totaro-Graham-Edidin model for the Morel-Voevodsky $B_\text{\'et} G$ as our $BG$. 

Fixing integers $n>0$, $m\ge0$, let $M_{n\times m+n}\cong \A^{n(m+n)}$ be the affine space of $m\times(m+n)$ matrices and  let $E_m\GL_n\subset M_{n\times m+n}$ be the open subscheme  of the  matrices of maximal rank $n$. We give $E_m\GL_n$ the base-point $x_0=(I_n, 0_n,\ldots, 0_n)$ where $I_n\in M_{n\times n}$ is the identity matrix and $0_n\in M_{n\times 1}$ is the 0-matrix. We let  $E\GL_n$ be the presheaf on $\Sm/k$
\[
E\GL_n=\colim_mE_m\GL_n
\]
which we may consider as a pointed presheaf with base-point $x_0$ if we wish. 
We note that $\GL_n$ acts freely on $E_m\GL_n$ for each $m\ge0$. Fixing an embedding $G\subset \GL_n$ for some $n$, we have the (pointed) quotient scheme 
\[
B_mG:=G\backslash E_m\GL_n
\]
and the (pointed) presheaf on $\Sm/k$
\[
BG:=\colim_mB_mG
\]

For $X\in \Sm/k$, we have
\[
BG(X)=\colim_mB_mG(X)
\]
where $B_mG(X)$ denotes the morphisms from $X$ to $B_mG$ in $\Sm_k$ and $BG(X)$ is the set of morphisms from the representable presheaf $X$ to $BG$.Thus we can speak of the $k$-points of $BG$. 

Although $B_mG$ is almost never the presheaf quotient of $E_m\GL_n$ by $G$, $B_mG$ is the quotient $G\backslash E_m\GL_n$ as \'etale sheaf, and if $G$ is  special, $B_mG$ is the quotient $G\backslash E_m\GL_n$ as Zariski sheaf. These facts pass to the morphism $E\GL_n\to BG$.  Although $BG$ is not in general the presheaf quotient, we nonetheless will often write $BG=G\backslash E\GL_n$. We note that the $\A^1$-homotopy type of $BG$ is independent of the choice of embedding $G\subset \GL_n$ and does in fact give a model for the Morel-Voevodsky $B_\text{\'et}G$ in the unstable motivic homotopy category \cite[Remark 4.27]{MorelVoev}.  

For each $m\ge0$, we have the principal $G$-bundle
\[
G\to E_m\GL_n\to B_mG
\]
which is \'etale locally trivial if $G$ is smooth over $k$ and is  Zariski locally trivial if $G$ is special. We refer to the sequence
\[
G\to E\GL_n\to BG
\]
as a principal $G$-bundle. 

For any $G$ defined as a closed subgroup-scheme of $\GL_n$ or of $\SL_n$, we will always use the embedding $G\subset \GL_n$ in our definition of $B_mG$ and $BG$, unless we explicitly say otherwise.

For $G=\GL_n$, we recall that $B_m\GL_n$ is just the Grassmann scheme $\Gr(n, n+m)$. For $G=N_T:=N_T(\GL_n)\subset \GL_n$ we have   the fiber bundle
\[
B_mN_T\to B_m\GL_n.
\]
Since $\GL_n$ is special the bundle $E_m\GL_n\to B_m\GL_n$ is a Zariski locally trivial principal $\GL_n$-bundle; this gives the description of $B_mN_T\to B_m\GL_n$ as
\[
B_mN =(N_T\backslash \GL_n)\times^{\GL_n}E_m\GL_n
\]
that is $\pi_{\GL_n,m}:B_mN_T\to B_m\GL_n$ is a Zariski locally trivial fiber bundle with fiber $F=N_T\backslash \GL_n$. 

Applying this construction to $N_T(\SL_n)\subset\SL_n\subset\GL_n$, gives us the Zariski locally trivial fiber bundle
\[
\pi_{\SL_n,m}: B_mN_T(SL_n)\to B_m\SL_n
\]
with fiber $N_T\backslash \GL_n$, noting that the inclusion $\SL_n\to\GL_n$ induces an isomorphism of schemes
\[
N_T(\SL_n)\backslash\SL_n\to N_T(\GL_n)\backslash\GL_n.
\]

\begin{theorem}\label{thm:BGSplitting} Let $k$ be a field of characteristic zero and let $\sE\in\SH(k)$ be a motivic ring  spectrum. Then the maps $\pi_{\GL_n}:B_mN_T(\GL_n)\to B_m\GL_n$ and $\pi_{\SL_n}:B_mN_T(\SL_n)\to B_m\SL_n$ induce split injections
\[
\pi_{\GL_n}^*:\sE^{*,*}(B_m\GL_n)\to \sE^{*,*}(B_mN_T(\GL_n))
\]
\[
\pi_{\SL_n}^*:\sE^{*,*}(B_m\SL_n)\to \sE^{*,*}(B_mN_T(\GL_n),
\]
for each $m\ge0$, and injections
\[
\pi_{\GL_n}^*:\sE^{*,*}(\BGL_n)\to \sE^{*,*}(BN_T(\GL_n)),
\]
\[
\pi_{\SL_n}^*:\sE^{*,*}(\BSL_n)\to \sE^{*,*}(BN_T(\SL_n)).
\]
\end{theorem}

\begin{proof} We give the proof for $\GL_n$, the proof for $\SL_n$ is the same; we write $N_T$ for $N_T(\GL_n)$.

Since $k$ has characteristic zero and $N_T\backslash \GL_n$ is smooth over $k$, $N_T\backslash \GL_n$ is strongly dualizable. By our computations above verifying the weak form of Conjecture~\ref{conj:RedConj}, the Euler characteristic $\chi(N_T\backslash \GL_n/k)$ is invertible in $\GW(k)$, thus by Theorem~\ref{thm:motBG}, the map
\[
\pi_m^*:\sE^{*,*}(B_m\GL_n)\to \sE^{*,*}(B_mN_T)
\]
is split injective, with splitting natural in $m$. 

This gives split injections
\[
\pi^*:R^i\lim_m\sE^{*,*}(B_m\GL_n)\to R^i\lim_m\sE^{*,*}(B_mN_T)
\]
for $i=0,1$, and by looking at the Milnor sequences
\[
\xymatrixcolsep{15pt}
\xymatrix{
0\ar[r]&R^1\lim_m\sE^{*-1,*}(B_m\GL_n)\ar[r]\ar[d]^{\pi^*}& \sE^{*,*}(\BGL_n)\ar[r]
\ar[d]^{\pi^*}& 
\lim_m\sE^{*,*}(B_m\GL_n)\ar[r]\ar[d]^{\pi^*}&0\\
0\ar[r]&R^1\lim_m\sE^{*-1,*}(B_mN_T)\ar[r]& \sE^{*,*}(BN_T)\ar[r]
& 
\lim_m\sE^{*,*}(B_mN_T)\ar[r]&0
}
\]
we see that $\pi_{\GL_n}^*:\sE^{*,*}(\BGL_n)\to \sE^{*,*}(BN_{T_{\GL_n}})$ is injective.
\end{proof}

We conclude this section with a proof of the strong form of the conjecture for $\GL_2$  (or $\SL_2$). We argue as above: $(\G_m)^2\backslash \GL_2=\P^1\times\P^1\setminus\Delta$, where $\Delta\cong \P^1$ is the diagonal, and thus $N_T\backslash \GL_2=\mathfrak{S}_2\backslash (\P^1\times\P^1\setminus\Delta)$. We have $\mathfrak{S}_2\backslash \P^1\times\P^1=\Sym^2\P^1\cong \P^2$. Via this isomorphism the quotient map
\[
\pi:\P^1\times\P^1\to \P^2=\Sym^2\P^1
\]
is given by symmetric functions,
\[
\pi((x_0: x_1),(y_0:y_1))=(x_0y_0:x_0y_1+x_1y_0:x_1y_1)
\]
Thus the restriction of $p$ to $\Delta$ is the map $\bar{\pi}(x_0:x_1)=(x_0^2: 2x_0x_1: x_1^2)$ with image the conic $C\subset \P^2$ defined by $Q:=T_1^2-4T_0T_2$. Furthermore, this identifies the double cover $\P^1\times\P^1\to \P^2$ with $\Spec\sO_{\P^2}(\sqrt{Q})$, that is, the closed subscheme of $O_{\P^2}(1)$ defined by pulling back the section of $O_{\P^2}(2)$ defined by $Q$ via the squaring map $O_{\P^2}(1)\to O_{\P^2}(2)$.

For $k$ of characteristic $\neq2$, the map $\bar{\pi}:\P^1\to C$ is an isomorphism. If $\Char k=2$, then $C$ is just the line $T_1=0$ and $N_T\backslash \GL_2\cong \A^2$. 

We have the $\A^1$- cofiber sequence
$\P^2\setminus C\hookrightarrow \P^2\to \Th(N_{C/\P^2})$.
The identity $\chi(Th(N_{C/\P^2})/k)=\<-1\>\chi(C/k)$ \cite[Proposition1.10(2)]{LevEnum} together with the  additivity of $\chi(-/k)$ in distinguished triangles gives
\[
\chi(N_T\backslash \GL_2/k)=\chi(\P^2\setminus C/k)=\chi(\P^2/k)-\<-1\>\chi(\P^1/k)
\]
Applying \eqref{eqn:EulerCharCell}(ii), this gives 
\[
\chi((N_T\backslash \GL_2)/k)=2\<1\>+\<-1\>-\<-1\>\cdot(\<1\>+\<-1\>)=\<1\>,
\]
verifying the strong form of the conjecture in this case.

\begin{proposition} For $k$ a perfect field, $\chi((N_T\backslash \GL_2)/k)=1$.
\end{proposition}

\section{Some background on Euler and Pontryagin classes}\label{sec:background} We now turn to a more detailed study of Euler classes and Pontryagin classes. In this section we recall some necessary background on characteristic classes, which we will use throughout the rest of the paper.  This material is all taken from \cite{Anan}.

We work over a fixed perfect base-field $k$, giving us the motivic stable homotopy category $\SH(k)$ as the setting for our constructions, 

We have the algebraic Hopf map $\eta_2:\A^2\setminus\{0\}\to \P^1$, giving the stable Hopf map $\eta:\Sigma_{\G_m}\mS_k\to \mS_k$, that is $\eta\in \pi_{1,1}(\mS_k)(k)$. 
For $\sE$ a motivic ring spectrum, the unit map gives us the element $\eta_\sE\in \sE^{-1,-1}(k)$.  $\sE$ is said to be an {\em $\eta$-invertible theory}  if multiplication by $\eta_\sE$ is invertible on $\sE^{**}(k)$, that is, there is an element $\eta_\sE^{-1}\in \sE^{1,1}(k)$ with $\eta_\sE\cdot \eta_\sE^{-1}=1_\sE$. This implies that multiplication by $\eta_\sE$ is an isomorphism on $\sE^{**}(F)$ for all $F\in \SH(k)$. 

We take the definition of {\em an $\SL$-orientation}  on $\sE$ from \cite[Definition 4]{Anan}, for us a an $\SL$-orientation is alway normalized in the sense of {\it loc. cit.}

Following \cite[Definition 1]{Anan}, an $\SL_n$-bundle on some $X\in \Sm/k$ is a vector bundle $E\to X$ together with an isomorphism $\theta:\det E\xrightarrow{\sim} \sO_X$. We use as well the definition of the Euler class $e(E, \theta)\in \sE^{2n, n}(X)$ of an $\SL_n$-bundle in an $\SL$-oriented theory $\sE$ \cite[Definition 4]{Anan},   and the Pontryagin classes $p_m(E)\in \sE^{4m, 2m}(X)$ of an arbitrary vector bundle $E\to X$, for an $\SL$-oriented, $\eta$-invertible theory $\sE$ \cite[Definition 19]{Anan}.

For an $\SL$-oriented theory $\sE$ and a line bundle $L\to X$, $X\in \Sm/k$, one has the twisted cohomology 
\[
\sE^{a,b}(X; L):=\sE^{a+2, b+1}_{0_L}(L) 
\]
The Euler class for $\SL_n$ bundles extends to a theory of Euler classes
\[
e(E)\in \sE^{2n,n}(X;\det^{-1}E)
\]
for arbitrary rank $n$ vector bundles $E\to X$, $X\in \Sm/k$ as follows: Let $L=\det E\to X$. One has a canonical trivialization of $\det(E\oplus L^{-1})$, giving the Thom class $\th(E\oplus L^{-1},can)\in \sE^{2n+2, n+1}_{0_{E\oplus L^{-1}}}(E\oplus L^{-1})$. Pulling back by the inclusion $i:L^{-1}=0_X\oplus L^{-1}\to E\oplus L^{-1}$ gives the class 
\[
e(E):=i^*\th(E\oplus L^{-1},can)\in  \sE^{2n+2, n+1}_{0_{L^{-1}}}(L^{-1})=
\sE^{2n, n}(X;\det^{-1}E).
\]
For details and properties of the twisted cohomology for $\SL$-oriented theories, we refer the reader to \cite[\S3]{LevRaksit} and \cite[\S3]{Anan19}.

For a vector bundle $E\to X$, we have the total Pontryagin class $p(E)=1+\sum_{i\ge1}p_i(E)\in \sE^{4*, 2*}(X)$, which satisfies the Whitney formula for $\SL$-bundles: if 
\[
0\to E'\to E\to E''\to 0
\]
is an exact sequence of $\SL$-vector bundles on $X\in \Sm/k$, then
\[
p(E)=p(E')\cdot p(E'')
\]
In addition, for an $\SL$ bundle $E$ of rank $n$, $p_i(E)=0$ for $2i>n$. These facts follow from \cite[Corollary 3, Lemma 15]{Anan}. For theories $\sE$ which are $\SL$-oriented and $\eta$-invertible, these facts hold for arbitrary vector bundles and arbitrary short exact sequences of vector bundles \cite[Corollary 7.9]{Anan19}.

\begin{remark}\label{rem:MWThy}  As explained in \cite{FaselSri}, the Eilenberg-MacLane spectrum associated to the graded sheaf of Milnor-Witt sheaves $\sK^{MW}_*$ admits Thom classes for bundles with a trivialized determinant, hence $\EM(\sK^{MW}_*)$ defines an $\SL$-oriented theory. Since $\sW=\sK^{MW}_*[\eta^{-1}]$,  the Eilenberg-MacLane spectrum $EM(\sW)$ is $\SL$-oriented and $\eta$ invertible.
\end{remark}

 \begin{remark} Ananyevskiy \cite[Corollary 5.4.]{Anan19} shows, that for $\sE\in \SH(k)$ a motivic ring spectrum such that $\sE^{0,0}$ is a Zariski sheaf on $\Sm/k$, $\sE$ admits a unique $\SL$ orientation.    
\end{remark}

We recall the following results of Ananyevskiy.

\begin{theorem}[\hbox{\cite[Theorem 10]{Anan}}]\label{thm:AnanBSLCoh} For $\sE$ an $\eta$-invertible, $\SL$-oriented theory, one has
\[
\sE^{*,*}(\BSL_n)=\begin{cases} \sE^{**}(k)[p_1,\ldots, p_{m-1}, e]&\text{ for }n=2m\\
 \sE^{**}(k)[p_1,\ldots, p_m]&\text{ for }n=2m+1
 \end{cases}
 \]
Here $p_i=p_i(E_n)$, $e=e(E_n)$, where $E_n\to \BSL_n$ is the universal rank $n$ bundle. Moreover $p_m(E_{2m})=e(E_{2m})^2$.
\end{theorem}
and  
\begin{theorem}[$\SL_2$-splitting principle] Let $i_m:(\SL_2)^m\to \SL_{2m}$ be the block-diagonal inclusion and let $\sE$ be an $\eta$-invertible, $\SL$-oriented theory. Then 
\[
i_m^*:\sE^{*,*}(\BSL_{2m})\to \sE^{*,*}((\BSL_2)^m)
\]
is injective. Moreover 
\[
 \sE^{*,*}((\BSL_2)^m)= \sE^{*,*}(\BSL_2)^{\otimes_{\sE^{**}(k)}m}=
\sE^{**}(k)[e_1,\ldots, e_m]
\]
where $e_i$ is the pull-back of $e\in \sE^{2,1}(\BSL_2)$ via the $i$th projection, and
\[
i_m^*e=e_1\cdot\ldots\cdot e_m;\quad i_m^*(1+\sum_ip_i)=\prod_{i=1}^r(1+e_i^2).
\]
\end{theorem}
This latter result is a consequence of \cite[Theorem 6, Theorem 10]{Anan}.

\section{$\BGL_n$ and $\BSL_n$}\label{sec:Det}

Let $\pi_n:\BSL_n\to \BGL_n$ be the canonical map, $E_n\to \BGL_n$ the universal rank $n$ vector bundle. For $\sE$ an $\SL$-oriented motivic ring spectrum, we have the pull-back map
\[
\pi^*:\sE^{**}(\BGL_n)\to \sE^{**}(\BSL_n)
\]
In addition, the bundle $\pi^*E_n$ is canonically isomorphic to the universal bundle $\tilde{E}_n\to \BSL_n$, hence there is a canonical trivialization $\theta:\pi^*\det^{-1} E_n\xrightarrow{\sim}\sO_{\BSL_n}$. We may therefore define
\[
\pi^*:\sE^{**}(\BGL_n; \det^{-1}E_n)\to \sE^{**}(\BSL_n)
\]
as the composition
\[
\sE^{**}(\BGL_n; \det^{-1}E_n)\xrightarrow{\pi^*}
\sE^{**}(\BSL_n(\pi_n^*\det^{-1}E_n))\xrightarrow{\theta_*}
\sE^{**}(\BSL_n).
\]
Our main theorem of this section is

\begin{theorem}\label{thm:BSLDecomp} Let $\sE\in \SH(k)$ be an $\SL$-oriented, $\eta$-invertible motivic ring spectrum. Suppose that either   $\sE^{0,0}$ is a Zariski sheaf on $\Sm/k$ or, that the unit map $\mS_k\to \sE$ makes $\sE^{0,0}$ a $\sW$-module. Then for $n=2m$, $\pi^*$ induces an isomorphism
\[
\pi^*:\sE^{**}(\BGL_n)\oplus \sE^{**}(\BGL_n; \det E_n)\to \sE^{**}(\BSL_n).
\]
More precisely, 
\begin{align*}
\pi^*(\sE^{**}(\BGL_n))&=\sE^{**}(k)[p_1,\ldots, p_{m-1}, e^2]\\
\pi^*(\sE^{**}(\BGL_n, \det E_n))&=e\cdot\sE^{**}(k)[p_1,\ldots, p_{m-1}, e^2].
\end{align*}

For $n=2m+1$, $\sE^{**}(\BGL_n; \det E_n))=0$ and 
\[
\pi^*:\sE^{**}(\BGL_n) \to \sE^{**}(\BSL_n)
\]
is an isomorphism.
\end{theorem}

\begin{proof} Write $\bar{\pi}:L\to \BGL_n$ for the line bundle $\det E_n$. We may identify the morphism $\pi:\BSL_n\to \BGL_n$ with the principal $\G_m$-bundle $L\setminus 0_L$ associated to $L\to \BGL_n$.   Identifying $\sE^{**}(\BGL_n)$ with $\sE^{**}(L)$ via $\bar{\pi}^*$ and using homotopy invariance gives us the localization sequences
\begin{multline*}
\ldots\to\sE^{a,b}(\BGL_n)\xrightarrow{\pi^*}\sE^{a,b}(\BSL_n)\\\to \sE^{a-1,b-1}(\BGL_n;L)
\xrightarrow{\phi_{a-1}} \sE^{a+1,b}(\BGL_n)\to\ldots
\end{multline*}
where $\phi_{a-1}$ is cup product with the Euler class $e(L)$. 

By Lemma~\ref{lem:PStab},  $e(L)=0$ and thus the localization sequence splits into short exact sequences
\[
0\to \sE^{a,b}(\BGL_n)\xrightarrow{\pi_n^*}\sE^{a,b}(\BSL_n)\xrightarrow{\del_{a,b}} \sE^{a-1,b-1}(\BGL_n;L)\to 0.
\]

Next, we define a splitting to $\del_{a,b}$. Let $t\in H^0(L,\bar{\pi}^*L)$ be the tautological section, giving a generator of $\bar{\pi}^*L$ over $L\setminus 0_L$. Let $t^\vee\in H^0(L\setminus 0_L,\bar{\pi}^*L^{-1})$ be the dual of $t$. This  defines the $\bar{\pi}^*L^{-1}$-valued quadratic form $q(s)=s^2\cdot t^\vee$, giving the section $\<t^\vee\>$  of the sheaf of twisted Witt groups  $W(L^{-1})$ over $L\setminus \{0_L\}$. 

Suppose  the unit map $\mS_k\to \sE$ makes $\sE^{**}(-)$ into a $\sW(-)$ module.  This gives us the section   $\<t^\vee\>_\sE:=\<t^\vee\>\cdot 1_\sE\in \sE^{0,0}(L\setminus\{0_{\BGL_n}\};L^{-1})$, that is $\<t^\vee\>_\sE\in \sE^{0,0}(\BSL_n, L^{-1})$.  One finds
\[
\del \<t^\vee\>_\sE=\eta\cdot 1_\sE\in 
\sE^{-1,-1}(\BGL_n)
\]
by making a computation in local coordinates. Indeed, we may first make the computation for $\sE=\EM(\sW)$, $\EM(\sW)^{a,b}(-)=H^{a-b}(-, \sW)$, in which case  $\sE^{-1,-1}(-)=H^0(-, \sW)$, justifying the local computation. In general, we have $\<t^\vee\>_\sE=\<t^\vee\>\cdot 1_\sE$ and so 
$\del \<t^\vee\>_\sE= \del \<t^\vee\>\cdot 1_\sE= \eta\cdot 1_\sE$. In addition, the map 
\[
\pi^*:\sE^{a,b}(\BGL_n; L)\to \sE^{a,b}(\BSL_n)
\]
is equal to the composition
\[
\sE^{a,b}(\BGL_n; L)\xrightarrow{\pi^*}\sE^{a,b}(\BSL_n; L)
\xrightarrow{\<t^\vee\>_\sE\cdot (-)}\sE^{a,b}(\BSL_n),
\]
which one sees by noting that multiplication by $t^\vee$ is the same as the canonical trivialization $\pi^*L\to \sO_{\BSL_n}$.   Thus 
\[
\del\circ \pi^*=\eta\cdot (-):\sE^{a,b}(\BGL_n; L)\to \sE^{a-1,b-1}(\BGL_n; L)
\]
which is an isomorphism, since $\sE$ is an $\eta$-invertible theory.

In case $\sE^{0,0}$ is a Zariski sheaf, as $\sE$ is $\eta$-invertible, $\sE^{-1,-1}$ is also a Zariski sheaf, so the computation of $\del \<t^\vee\>_\sE$ may be made locally as above, with the same result.
 
It remains to compute the image of the two summands. For $n=2m+1$, the map
\[
\pi^*:\sE^{**}(\BGL_n)\to \sE^{**}(\BSL_n).
\]
is surjective: since the Pontryagin classes are defined for arbitrary vector bundles, and are functorial,  we have
\[
p_i=\pi^*(p_i(E_n));\quad i=1,\ldots, m. 
\]
The surjectivity then follows from Ananyevskiy's description of $\sE^{**}(\BSL_n)$, and this implies that the remaining summand $\sE^{*, *}(\BGL_n; L)$ must be zero.

Now take $n=2m$. We  write $R_{m-1}$ for the subring $\sE^{**}(k)[p_1,\ldots, p_{m-1}]$ of 
$\sE^{**}(\BSL_n)$.  The same argument  as above together with the identity
\[
p_m(\tilde{E}_{2m})=e(\tilde{E}_{2m})^2
\]
implies that $\pi_n^*(\sE^{**}(\BGL_n))$ contains  $R_{m-1}[e^2]$, and  $e$ is in $\pi_n^*(\sE^{**}(\BGL_n; L))$. Since $\sE^{**}(\BGL_n; L)$ is a module for $\sE^{**}(\BGL_n)$, this shows that  the $R_{m-1}[e^2]$-submodule
$e\cdot R_{m-1}[e^2]$ of $\sE^{**}(\BSL_n)$ is contained in $\pi^*_n(\sE^{**}(\BGL_n; L))$. Since
\[
\sE^{**}(\BSL_n)=R_{m-1}[e^2]\oplus e\cdot R_{m-1}[e^2],
\]
the result for $n=2m$ follows.
\end{proof}

\begin{remark}
Since  the Eilenberg-MacLane spectrum $\EM(\sW)$ is $\SL$-oriented and $\eta$ invertible  we may apply theorem~\ref{thm:BSLDecomp} to give a computation of $H^*(\BGL_n, \sW)$ and $H^*(\BGL_n, \sW(\det E_n))$: For $n=2m$, 
\begin{align*}
&H^*(\BGL_n, \sW)=W(k)[p_1,\ldots, p_{m-1}, p_m], \\
&H^*(\BGL_n, \sW(\det E_n))=e\cdot W(k)[p_1,\ldots, p_{m-1}, p_m]
\end{align*}
and for $n=2m+1$, 
\begin{align*}
&H^*(\BGL_n, \sW)=W(k)[p_1,\ldots, p_{m-1}, p_m], \\
&H^*(\BGL_n, \sW(\det E_n))=0.
\end{align*}
\end{remark}

We used the following result, a slight extension of a result of Ananyevskiy's \cite[Theorem 7.4]{Anan19}, in the proof of Theorem~\ref{thm:BSLDecomp}. 

\begin{lemma}\label{lem:PStab} Let   $\sE\in \SH(k)$ be an $\SL$-oriented, $\eta$-invertible motivic ring spectrum. Then for $X\in \Sm/k$ and  $V\to X$ a vector bundle of odd rank $r$, we have $e(V)=0$ in $\sE^{2r,r}(X;\det V^{-1})$.
\end{lemma}

\begin{proof} Let $L=\det^{-1}V$. By definition, $\sE^{2r,r}(X;L)=\sE^{2r+2,r+1}_{0_L}(L)$. We have the canonical trivialization $\lambda:\det(V\oplus L)\to \sO_X$, giving the Thom class $\th(V\oplus L,\lambda)\in \sE^{2r+2, r+1}_{0_V\times 0_L}(V\oplus L)$, and the Euler class $e(V)\in \sE^{2r,r}(X;L)$ is by definition the pullback $s_{0_V}^*(\th(V\oplus L,\lambda))\in \sE^{2r+2,r+1}_{0_L}(L)$, where $s_{0_V}:L\to V\oplus L$ is the closed immersion $s_{0_V}(x)=(0_V,x)$. 

Let $\pi:X\times \A^1\setminus\{0\}\to X$ be the projection, let $t$ be the standard unit on $\A^1\setminus\{0\}$ and let $\rho:\pi^*(V\oplus L)\to \pi^*(V\oplus L)$ be the automorphism $t\cdot \id_{\pi^*V}\oplus\id_{\pi^*L}$. We also use $\pi$ to denote the induced projections $\pi^*V\to V$, $\pi^*L\to L$, and so on, and let $\pi^*s_{0_V}:\pi^*L\to \pi^*(V\oplus L)$ denote the map induced by $s_{0_V}$ . The proof of  \cite[Theorem 7.4]{Anan19} gives the identity
\[
\<t^{2r+1}\>\cdot \th(\pi^*(V\oplus L),\pi^*\lambda)=\rho^*\th(\pi^*(V\oplus L),\pi^*\lambda)
\]
and since $\rho\circ \pi^*s_{0_V}=\pi^*s_{0_V}$, this implies
\[
\<t\>\cdot \pi^*e(V)=\<t^{2r+1}\>\cdot \pi^*e(V)=\pi^*e(V)\in \sE^{2r+2,r+1}_{0_L\times \A^1\setminus\{0\}}(L\times \A^1\setminus\{0\}).
\]

Following  Ananyevskiy's argument (where the minus sign disappears since we are using the ``alternate'' normalization mentioned in \cite[Remark 6.5]{Anan19}),  applying the boundary in the localization sequence
\begin{multline*}
\ldots\to\sE^{2r, r}(X\times \A^1\setminus\{0\}; \det^{-1}V)\xrightarrow{\delta}\sE^{2r-1, r-1}(X; \det^{-1}V)\\\to \sE^{2r+1, r+1}(X\times\A^1;\det^{-1}V)\to\ldots
\end{multline*}
gives the identity
\[
0=\delta((\<t\>-1)\cdot\pi^*e(V))=\eta\cdot e(V)
\]
in $\sE^{2r-1, r-1}(X; \det^{-1}V)$. Since $\sE$ is by assumption an $\eta$-invertible theory, we have $e(V)=0$.
\end{proof}

\section{Witt cohomology of $BN_T(\SL_2)$} \label{secWittCoh}

In this section, for $V\to X$ a vector bundle, we let $f:\P(V)\to X$ denote the associated projective space bundle $\Proj_{\sO_X}(\Sym^* V^\vee)$. We have  the tautological subbundle $O(-1)\to \P(V)$ of $f^*V$ and its dual $O(1)$.

We have the $\SL$-oriented theory $\EM(\sW)^{a,b}(X):=H^{a-b}(X, \sW)$ on $\Sm/k$, where $\sW$ is the sheaf of Witt rings. Our main goal in this section is to compute  $H^*(BN_T(\SL_2), \sW)$, where $N_T(\SL_2)\subset \SL_2$ is the normalizer of the standard torus 
\[
T:=\left\{\begin{pmatrix}t&0\\0&t^{-1}\end{pmatrix}\right\}\subset \SL_2
\]
We write $N:=N_T(\SL_2)$. $N$ is generated by $T$ and an additional element $\sigma=\begin{pmatrix}0&1\\-1&0\end{pmatrix}$, and sending $\sigma$ to $-1$ gives us the exact sequence
\[
1\to T\to N\to \{\pm1\}\to0
\]
This gives us the representation $\rho^-:N\to \G_m$ sending $T$ to $1$ and $\sigma$ to $-1$ (see \S\ref{sec:Rank2Bundles} below).

We work over a field $k$ of characteristic $\neq2$. In this case, the computation of section~\ref{sec:conj} gives the isomorphism
\[
N\backslash\SL_2\cong N_T(\GL_2)\backslash\GL_2\cong\P^2\setminus C
\]
with $C$ the conic defined by $Q:=T_1^2-4T_0T_2$.

The natural right $\SL_2$-action on $N\backslash\SL_2$ can be described in the model $\P^2\setminus C$ as follows.
 Let $F=\A^2$ with the standard (right) $\SL_2$-action. We have the canonical map of schemes
\[
sq:\P(F)\to \P(\Sym^2F)
\]
induced by the squaring map $sq:F\to \Sym^2F$, where  $sq(v)$ is the image of $v\otimes v$ under the canonical surjection $F\otimes F\to \Sym^2F$. As the right $\SL_2$-action on $\Sym^2F$ is induced by the diagonal action on $F\otimes F$, $(v\otimes w)\cdot g:=v\cdot g\otimes w\cdot g$, the map $sq$ is $\SL_2$-equivariant. 

This identifies $C\subset \P^2$ with $sq((\P(F))\subset \P(\Sym^2F)$.  Furthermore, the quadratic polynomial $Q=T_1^2-4T_0T_2$ gives us a global section $\<Q\>$ of $\sO_{\P^2\setminus C}^\times/\sO_{\P^2\setminus C}^{\times\ 2}$,   with value $Q/L^2$ on $\P^2\setminus (C\cup\{L=0\})$ for each linear form $L$. As the map sending a unit $u$ to the class of the rank one quadratric form $\<u\>$ descends to a map of sheaves $\sO^\times/\sO^{\times 2}\to \sGW^\times$, we have the section $\<q\>$ of $\sGW$ on $\P^2\setminus C$ corresponding to $\<Q\>$.

Since $C$ is stable under the $\SL_2$ action, the form $Q$ is $\SL_2$ invariant up to a scalar, but as 
\[
g\mapsto \frac{g^*Q}{Q}
\]
defines a character of $\SL_2$, $Q$ is $\SL_2$ invariant and thus $\<Q\>\in H^0(\P^2\setminus C,\sO^\times/\sO^{\times 2})$ is also $\SL_2$ invariant.

We may use the model 
\[
(N\backslash\SL_2)\times^{\SL_2}E\SL_2\cong N\backslash E\SL_2
\]
for $BN$.  The  $\SL_2$-invariant section $\<Q\>$ of $\sO_{N\backslash\SL_2}^\times/\sO_{N\backslash\SL_2}^{\times\ 2}$  on $N\backslash\SL_2$ thus descends to a global section  $\<\bar{Q}\>$ of $\sO^\times/\sO^{\times\ 2}$ on $BN$,  which gives us the global section $\<\bar{q}\>$ of the sheaf of Grothendieck-Witt rings $\sGW$ on $BN$. We will also denote the image of $\<\bar{q}\>$ to $H^0(BN,\sW)$ by $\<\bar{q}\>$, using the context to make the meaning clear.

 Since $\SL_2$ is special, the bundle $E\SL_2\to \BSL_2$ is Zariski locally trivial, and thus the same holds for $BN\to \BSL_2$, as explained in section~\ref{sec:conj}.

Our description of $N\backslash\SL_2$ as $\P(\Sym^2F)\setminus sq(\P(F))$ realizes $BN$ as an open subscheme of the $\P^2$-bundle $\P(\Sym^2F)\times^{\SL_2}E\SL_2\to \BSL_2$ with complement the $\P^1$-bundle $\P(F)\times^{\SL_2}E\SL_2\to \BSL_2$:
\[
\xymatrix{
BN\ar[r]^-j\ar[dr]_p&\P(\Sym^2F)\times^{\SL_2}E\SL_2\ar[d]^{\tilde{p}}&\ar[l]_-i
\P(F)\times^{\SL_2}E\SL_2\ar[dl]^{\bar{p}}\\
&\BSL_2
}
\]

The representation $\rho^-:N\to \G_m$ gives us the line bundle   $\gamma$ on $BN$, generating $\Pic (BN)$.  We have $\Pic(N\backslash\SL_2)=\Z/2$ with generator the restriction of $O_{\P^2}(1)$, and $\gamma$ is the same as the bundle induced by the restriction of $p_1^*O_{\P^2}(1)$ via the canonical $\SL_2$-linearization of $O_{\P^2}(1)$. Similarly, we have the bundle $p_1^*O(m)$ on $\P(F)\times^{\SL_2}E\SL_2$, induced by the  canonical $\SL_2$-linearization of $O_{\P^1}(1)$.

\begin{lemma} \label{lem:Witt1} 
\[
H^n(\P(F)\times^{\SL_2}E\SL_2,\sW)=\begin{cases} W(k)&\text{ for }n=0\\0&\text{ else.}\end{cases}
\]
\end{lemma}

\begin{proof} The element $\eta\in K^{MW}_{-1}(k)$ is realized in $\SH(k)$ as the stable Hopf map, the unstable version of which is the map $\eta_2:\A^2\setminus\{0\}\to \P^1$, $\eta_2(x,y)=[x:y]$.   As $\sW=\sK^{MW}_*[\eta^{-1}]$,  the projection
\[
\eta_2\times\id:(\A^2\setminus\{0\})\times^{\SL_2}E\SL_2\to
\P(F)\times^{\SL_2}E\SL_2
\]
induces an isomorphism
\[
\eta_2^*:H^*(\P(F)\times^{\SL_2}E\SL_2,\sW)\to H^*((\A^2\setminus\{0\})\times^{\SL_2}E\SL_2,\sW).
\]
Here we give $\A^2\setminus\{0\}$ the right $\SL_2$-action induced from the standard right action on $\A^2=F$.  

Giving $\SL_2$ the right $\SL_2$-action via right multiplication, we have the $\SL_2$-equivariant map
\[
r:\SL_2\to \A^2\setminus\{0\};\quad r\begin{pmatrix}a&b\\c&d\end{pmatrix}=(a,b)
\]
making $\SL_2$ an $\A^1$-bundle over $\A^2\setminus\{0\}$. Thus, by homotopy invariance,
\[
r^*:H^*((\A^2\setminus\{0\})\times^{\SL_2}E\SL_2,\sW)
\to H^*(\SL_2\times^{\SL_2}E\SL_2,\sW)=H^*(E\SL_2,\sW)
\]
is also an isomorphism. Since $E\SL_2$ is $\A^1$-contractible, we have
\[
H^n(E\SL_2,\sW)=\begin{cases} W(k)&\text{ for }n=0\\0&\text{ else,}\end{cases}
\]
and the lemma is proven.
\end{proof} 

Let $T\to \P(\Sym^2F)$ be the tangent bundle. The $\SL_2$-action on  $\Sym^2F$ extends canonically to an action on $T$ over its action on $\P(\Sym^2F)$, and $T\times^{\SL_2}E\SL_2\to 
\P(\Sym^2F)\times^{\SL_2}E\SL_2$ is the relative tangent bundle of $\P(\Sym^2F)\times^{\SL_2}E\SL_2$ over $BN$, which we denote by $\sT$.

\begin{lemma} \label{lem:Witt2} 1. The map
\[
\pi^*:H^*(\BSL_2,\sW)\to H^*(\P(\Sym^2F)\times^{\SL_2}E\SL_2,\sW)
\]
is an isomorphism. \\[3pt]
2. $H^*(\P(\Sym^2F)\times^{\SL_2}E\SL_2,\sW(p_1^*O_{\P^2}(1)))$  is a free $H^*(\BSL_2,\sW)$-module with degree two generator $e(\sT)$.
\end{lemma}

\begin{proof} 
(1) We use the Leray spectral sequence
\[
E_1^{p,q}=\oplus_{x\in \BSL_2^{(p)}}H^q(\pi^{-1}(x),\sW)\Rightarrow H^{p+q}(\P(\Sym^2F)\times^{\SL_2}E\SL_2,\sW).
\]
As $\P(\Sym^2F)\times^{\SL_2}E\SL_2\to \BSL_2$ is Zariski locally trivial, it follows that $\pi^{-1}(x)\cong \P^2_{k(x)}$ for all $x\in \BSL_2$. Fasel \cite[theorem 11.7]{FaselPI} computes $H^*(\P^2_{k(x)}, \sW)$ as
\[
H^n(\P^2_{k(x)}, \sW)=\begin{cases}W(k(x))&\text{ for }n=0\\0&\text{ else.}
\end{cases}
\]
Feeding this into the Leray spectral sequence gives the result.

For (2), we do the same, using the Leray spectral sequence converging to 
$H^{p+q}(\P(\Sym^2F)\times^{\SL_2}E\SL_2,\sW(p_1^*O(1))$ and the isomorphism
\[
H^n(\P^2_{k(x)}, \sW(O(1)))=\begin{cases}W(k(x))&\text{ for }n=2\\0&\text{ else,}
\end{cases}
\]
again by \cite[theorem 11.7]{FaselPI}.  The Gysin map $i_{p*}:H^{n-2}(\Spec k(x), \sW)\to  (\P^2_{k(x)}, \sW(\omega_{\P^2}))$ is left inverse to the pushforward  $\pi_*:H^n(\P^2_{k(x)}, \sW(\omega_{\P^2}))\to H^{n-2}(\Spec k(x), \sW)$; it is not hard to see that 
for $n=2$, these are the  isomorphism and its inverse, as given by Fasel's computation.  As $\omega_{\P^2_{k(x)}/k(x)}=O_{\P^2}(-3)$, we may replace $\sW(\omega_{\P^2})$ with $\sW(O(1)))$. 

The motivic Euler characteristic $\chi(\P^2_{k(x)}/k(x))\in \GW(k(x))$ has been computed by Hoyois as $2\<1\>+\<-1\>$, hence its image in $W(k)$ is $\<1\>$. We have the Gau{\ss}-Bonnet formula \cite[Theorem 1.5]{LevRaksit}, which says that $\chi(\P^2_{k(x)}/k(x))=\pi_*(e(T_{\P^2_{k(x)}}))$. This in turn implies that $e(T_{\P^2_{k(x)}})$ is a $W(k)$-generator for $H^2(\P^2_{k(x)}, \sW(\omega_{\P^2}))$ (one can also show this by computing the local indices associated to a section of $T_{\P^2}$ with isolated zeros, we leave the details of this to the reader). This shows that the Euler class of the relative tangent bundle $e(\sT)\in H^2(\P(\Sym^2F)\times^{\SL_2}E\SL_2,\sW(p_1^*O(1))$ gives a global class restricting to the generators in the $E_1$-term, and thus gives a generator for $H^*(\P(\Sym^2F)\times^{\SL_2}E\SL_2,\sW(p_1^*O_{\P^2}(1)))$ as a free $H^*(\BSL_2,\sW)$-module.
\end{proof}

Recall we have the closed immersion
\[
i:\P(F)\times^{\SL_2}E\SL_2\to \P(\Sym^2F)\times^{\SL_2}E\SL_2
\]
with open complement
\[
j:BN\to \P(\Sym^2F)\times^{\SL_2}E\SL_2.
\]

\begin{proposition}\label{prop:WittCoh1} 1. The map $p^*:H^n(\BSL_2,\sW)\to H^n(BN,\sW)$ is an isomorphism for all $n>0$. For $n=0$, we have the split exact sequence
\[
0\to H^0(\BSL_2,\sW)\xrightarrow{p^*}H^0(BN,\sW)\to W(k)\to 0
\]
with the splitting given by the section $\<\bar{q}\>\in H^0(BN,\sW)$. Explicitly
\[
H^*(BN,\sW)=W(k)[p^*e]\oplus W(k)\cdot\<\bar{q}\>.
\]
2. $H^*(BN, \sW(\gamma))$ is a free $W(k)[p^*e]$-module with generator $j^*e(\sT)$.
\end{proposition}

\begin{proof} We recall Ananyevskiy's computation of $H^*(\BSL_2, \sW)$  (Theorem~\ref{thm:AnanBSLCoh}): 
\[
H^*(\BSL_2,\sW)=W(k)[e]
\]
where $e=e(E_2)$ is the Euler class of the tautological rank two vector bundle on $\BSL_2$.

Since $\Pic(\BSL_2)=\{1\}$, it follows that $\Pic(\P(\Sym^2F)\times^{\SL_2}E\SL_2)=\Z[\sO(1)]$, and thus the normal bundle of $i$ is the bundle $p_1^*O_{\P(F)}(2)$ on $\P(F)\times^{\SL_2}E\SL_2)$. Moreover,   $i^*p_1^*O_{\P(\Sym^2F)}(m)=p_1^*O_{\P(F)}(2m)$ for all $m$. This gives us the localization sequence
\begin{multline*}
\ldots\to H^{n-1}(\P(F)\times^{\SL_2}E\SL_2,\sW)\\\xrightarrow{i_*}  H^n(\P(\Sym^2F)\times^{\SL_2}E\SL_2,\sW(p_1^*O_{\P(\Sym^2F)}(m)))\\\xrightarrow{j^*} H^n(BN,\sW(\gamma^m)
\xrightarrow{\delta} H^n(\P(F)\times^{\SL_2}E\SL_2,\sW)\to\ldots 
\end{multline*}
The result for $n\ge 1$ and for $n=0$ and $m=1$ follows from this and the lemmas \ref{lem:Witt1} and \ref{lem:Witt2}. For $n=0$ and $m=0$, one can check by looking at a single fiber $\pi^{-1}(x)$ that $\delta(\<\bar{q}\>)$ is a  $W(k)$-generator for $H^0(\P(F)\times^{\SL_2}E\SL_2,\sW)=W(k)$. 

\end{proof}

To conclude, we compute the multiplicative structure for 
$H^*(BN,\sW)$ and the $H^*(BN,\sW)$-module structure for $H^*(BN,\sW(\gamma))$. As the class $\<\bar{q}\>\in H^0(BN,\sW)$ is represented by a 1-dimensional form, we have $\<\bar{q}\>^2=1$. It remains to compute $\<\bar{q}\>\cdot p^*e\in H^2(BN,\sW)$ and $\<\bar{q}\>\cdot j^*e(\sT)\in H^2(BN,\sW(\gamma))$.

For this, we consider the inclusion of the torus $\G_m\to N$ and the associated map of classifying spaces $B\G_m\to BN$, where we use the model $(\G_m\backslash \SL_2)\times^{\SL_2}E\SL_2$ for $B\G_m$. We recall that $BN=(\P(\Sym^2F)\setminus \P(F))\times^{\SL_2}E\SL_2$. The $\SL_2$-invariant section $Q$ of $\sO_{\P(\Sym^2F)}(2)$ gives us the the section $p_1^*Q$ of $p_1^*\sO(2)$. We claim that the corresponding double cover $\Spec_{\sO_{BN}}(\sqrt{p_1^*Q})\to BN$ is isomorphic to $B\G_m\to BN$.

To see this, we note the isomorphism of schemes 
\[
\G_m\backslash\SL_2\cong  
\P^1\times\P^1\setminus\Delta;\quad \begin{pmatrix}a&b\\c&d\end{pmatrix}\mapsto ([a:b], [c:d])
\]
which is an $\SL_2$-equivariant isomorphism for the right action of $\SL_2$
on $\P^1\times\P^1\setminus\Delta$ induced by the right action on $F$.
Since the cover $\P^1\times\P^1\setminus\Delta\to \P(\Sym^2F)\setminus \P(F)$ is isomorphic to $\Spec_{\sO_{\P(\Sym^2F)\setminus \P(F)}}(\sqrt{Q})$, equivariantly for the right $\SL_2$-actions, the assertion follows.

\begin{lemma}\label{lem:Relation} $(1+\<\bar{q}\>)\cdot p^*e=0$ in $H^2(BN,\sW)$ and 
$(1+\<\bar{q}\>)\cdot j^*e(\sT)=0$ in  $H^2(BN,\sW(\gamma)$. 
\end{lemma}

\begin{proof} The relation $(1+\<\bar{q}\>)\cdot p^*e=0$ implies that
$(1+\<\bar{q}\>)\cdot j^*e(\sT)=0$. Indeed, since $H^*(BN,\sW(\gamma)$ is a free $W(k)[p^*e]$-module, there is a unique element $\lambda\in W(k)$ with $(1+\<\bar{q}\>)\cdot j^*e(\sT)=\lambda\cdot j^*e(\sT)$. Then $0=(1+\<\bar{q}\>)\cdot p^*e)\cdot j^*e(\sT)=\lambda\cdot (p^*e\cdot j^*e(\sT))$, and as $p^*e\cdot j^*e(\sT)$ is a generator for the free $W(k)$-module  $H^4(BN,\sW(\gamma)$,  $\lambda=0$. We now show that $(1+\<\bar{q}\>)\cdot p^*e=0$.

We have the tautological rank two bundle $E_2\to \BSL_2$, $E_2:=F\times^{\SL_2}E\SL_2$ and its pull-back $\tilde{E}_2\to BN$. The class $e\in H^2(\BSL_2,\sW)$ is the Euler class $e(E_2)$, so $p^*e=e(\tilde{E}_2)$. We have the cover $\phi:B\G_m\to BN$ with fiber $\G_m\backslash N=\Z/2$ corresponding to the inclusion $\G_m\to N$. The pullback $\phi^*\tilde{E}_2$ splits as a direct sum
\[
\phi^*\tilde{E}_2=O(1)\oplus O(-1)\xrightarrow{\pi}B\G_m
\]
corresponding to the decomposition of $F$ into the $t$ and $t^{-1}$ eigenspaces for the torus $\G_m\subset \SL_2$. The $\Z/2$-action on $B\G_m$ corresponds to the automorphism $t\mapsto t^{-1}$ of $\G_m$ and exchanges the two factors $O(1)$ and $O(-1)$. By descent, the union of the two ``coordinate axes'' $O(1)\times 0\cup 0\times O(-1)\subset \phi^*\tilde{E}_2$ gives us the cone $\mathfrak{C}\subset \tilde{E}_2$, which on each geometric fiber is the union of two lines through the origin. The normalization $\nu:\mathfrak{C}^N\to \mathfrak{C}$ is thus isomorphic to the total space of the bundle $O(1)$, that is, we have the commutative diagram
\[
\xymatrix{
\mathfrak{C}^N\ar[r]^\sim\ar[dr]_\nu&O(1)\ar[r]^{\tilde{i}}\ar[d]_{\tilde \phi}&\phi^*\tilde{E}_2\ar[r]\ar[d]_{\tilde{\phi}}&B\G_m\ar[d]^\phi\\
&\mathfrak{C}\ar[r]_i&\tilde{E}_2\ar[r]&BN
}
\]

Let $\sI_{\mathfrak{C}}$ denote the ideal sheaf of $\mathfrak{C}$ in $\tilde{E}_2$, and define $\sI_{O(1)}$ similarly. The normal bundle $N_{O(1)/\phi^*\tilde{E}_2}$ of the subscheme $O(1)$ in $\phi^*\tilde{E}_2$ is $\pi^*O(-1)_{|O(1)}$. The restriction of $\pi^*O(1)_{|O(1)}$ to $O(1)\setminus \{0_{O(1)}\}$ is trivialized by the tautological section $s$ of 
$\pi^*O(1)_{|O(1)}$ over $O(1)$, which induces a nowhere vanishing dual section $\partial/\partial s$ of the dual $N_{O(1)/\phi^*\tilde{E}_2}$ over $O(1)\setminus \{0_{O(1)}\}$. This gives us the section $\<\partial/\partial s\>$ of the twisted Grothendieck-Witt sheaf $\sGW(O(-1))$ over $O(1)\setminus \{0_{O(1)}\}$, as the class of the $O(-1)$-valued quadratic form $q_s$ with $q_s(1)=\partial/\partial s$. 

The boundary $\partial_{0_{\phi^*\tilde{E}_2}}(\<\partial/\partial s\>)$  lives in $H^0(0_{\phi^*\tilde{E}_2}, \sK^{MW}_{-1}(\det(\phi^*\tilde{E}_2))$; the canonical trivialization of $\det E_2$ sets this equal to $H^0(0_{\phi^*\tilde{E}_2}, \sK^{MW}_{-1})$. We claim that a calculation in local coordinates gives the identity
\[
\partial_{0_{\phi^*\tilde{E}_2}}(\<\partial/\partial s\>)=\eta\in H^0(0_{q^*\tilde{E}_2}, \sK^{MW}_{-1}).
\]
Indeed, take a local generator $\lambda$ for $O(1)$, giving the dual local generator $\lambda^\vee$ for $O(-1)=O(1)^\vee$. This gives a local  isomorphism of $\phi^*\tilde{E}_2$ with 
\[
\A^1\times \A^1\times B\G_m=\Spec k[x,y]\times_kB\G_m,
\]
 with $O(1)=\A^1\times 0\times B\G_m$, $O(-1)=0\times\A^1\times B\G_m$. This also identifies $N_{O(1)/\phi^*\tilde{E}_2}$ with the trivial bundle on $\A^1\times 0\times B\G_m$ and the section $\<\partial/\partial s\>$ goes over to $\<1\>\otimes x^{-1}=p_1^*\<x\>$ and the canonical generator of the normal bundle of $0_{O(1)}$ in $O(1)$ goes over to $\del/\del x$. On $\A^1=\Spec k[x]$, the element $\<x\>\in K^{MW}_0(k[x])$  has boundary 
\[
 \del_{\del/\del x}(\<x\>)=\del_{\del/\del x}(1+\eta[x])=\eta
 \]
 since $\del_{\del/\del x}([x])=1$.  Under the canonical identification $\sK^{MW}_{-1}\cong \sW$, $\eta$ goes over to $\<1\>$. We refer the reader to \cite[\S6]{MorelICTP} or \cite[\S6]{MorelNewton} for details on the Milnor-Witt sheaves and their relation to $\sGW$ and $\sW$.
 
Identifying $\mathfrak{C}\setminus\{0_{\tilde{E}_2}\}$ with $O(1)\setminus \{0_{O(1)}\}$ via $\phi$, we have the section $\phi_*(\<\partial/\partial s\>)$ of $\sW(N_{\mathfrak{C}\setminus\{0_{\tilde{E}_2}\}/\tilde{E}_2})$ over $\mathfrak{C}\setminus\{0_{\tilde{E}_2}\}$. This has boundary
\[
\del(\tilde{\phi}_*(\<\partial/\partial s\>))=\tilde{\phi}_*(\del(\<\partial/\partial s\>))=s_{0*}(\phi_*\<1\>)\in H^0(0_{\tilde{E}_2}, \sW)
\]
where $s_0:BN\to \tilde{E}_2$ is the 0-section. Since $\phi:B\G_m\to BN$ is the double cover $\Spec_{\sO_{BN}}(\sO_{BN}(\sqrt{p_1^*Q}))$, $\phi_*\<1\>$ is just the image in $H^0(BN, \sW)$ of the trace form of this double cover,  which gives us the element $\<2\>+\<2\bar{q}\>=\<2\>(1+\<\bar{q}\>)$.

Let $\tilde{E}_{2,m}$ denote the pull-back to $B_m\SL_2\subset \BSL_2$. We recall that $H^*(\tilde{E}_{2,m}, \sW)$ may be computed as the cohomology of the Rost-Schmid complex (see \cite[Chapter 5]{MorelA1Top})
\begin{multline*}
\sW(k(\tilde{E}_{2,m}))\to \oplus_{x\in \tilde{E}_{2,m}^{(1)}}\sW(k(x), \det N_{x/\tilde{E}_{2,m}})\\\to\oplus_{x\in \tilde{E}_{2,m}^{(2)}}\sW(k(x), \det N_{x/\tilde{E}_{2,m}})\to\ldots
\end{multline*}
To simplify the notation, we replace $\tilde{E}_{2,m}$ with $\tilde{E}_2$, with the understanding that all objects should be considered as  living on  $\tilde{E}_{2,m}$ for arbitrary $m$. 

We note that  $s_{0*}(1)\in H^2(\tilde{E}_2,\sW)$ is represented in the Rost-Schmid complex by the constant section $1$ on $0_{\tilde{E}_2}$. Computing  $H^2(\tilde{E}_2,\sW)$ via the Rost-Schmid complex, we see that 
\[
\del(\phi_*(\<\partial/\partial s\>))=\<2\>(1+\<\bar{q}\>) \in H^0(0_{\tilde{E}_2},  \sW) 
\]
goes to zero in $H^2(\tilde{E},\sW)$   As  $p^*e=e(\tilde{E}_2)=s_0^*s_{0*}(1)$, and as $s_0^*:H^2(\tilde{E}_2, \sW)\to H^2(BN, \sW)$ is an isomorphism, we see that  $\<2\>(1+\<\bar{q}\>)\cdot p^*e=0$ in $H^2(BN, \sW)$; since $\<2\>$ is a unit in $\sW$, it follows that $(1+\<\bar{q}\>)\cdot p^*e=0$.
\end{proof}

\begin{proposition} Let $W(k)[x_0, x_2]$ be the graded polynomial algebra over $W(k)$ on generators $x_0, x_2$, with $\deg x_i=i$. Then sending $x_0$ to $\<\bar{q}\>$ and $x_2$ to $p^*e$ defines  $W(k)$-algebra isomorphism
\[
\psi:W(k)[x_0, x_2]/(x_0^2-1, (1+x_0)x_2)\to H^*(BN, \sW).
\]
Moreover,  $H^*(BN, \sW(\gamma))$ is the quotient of the free $H^*(BN, \sW)$-module on the generator $j^*e(\sT)$ modulo the relation $(1+\<q\>)j^*e(\sT)=0$.
\end{proposition}

\begin{proof} It follows from Lemma~\ref{lem:Relation} and the fact that $\<\bar{q}\>$ is  represented locally by a one-dimensional form  that  $\psi$ is a well-defined $W(k)$-algebra homomorphism. 
 It follows from Proposition~\ref{prop:WittCoh1} that $\psi$ is a $W(k)$-module isomorphism, hence a $W(k)$-algebra isomorphism. The description of $H^*(BN, \sW(\gamma))$ as an $H^*(BN, \sW)$-module then follows from Proposition~\ref{prop:WittCoh1} and Lemma~\ref{lem:Relation}.
\end{proof}

As $Q$ is a section of $\sO_{\P(\Sym^2F)}(2)$, $p_1^*Q$ has a well-defined value in $k/k^{\times 2}$ at each point of $BN(k)$. Here we use the definition of $BN$ as an Ind-scheme to define $BN(k)$.

\begin{lemma} \label{lem:ThomIso} Let $x\in BN(k)$ be a $k$-point such that $p_1^*Q(x)=-1\mod k^{\times 2}$, and let $\tilde{E}_{2x}$ denote the fiber of $\tilde{E}_2$ at $x$ with origin $0_x$.  Let 
\[
s_{0*}:H^0(BN, \sW)\to H^2_{0_{\tilde{E}_2}}(\tilde{E}_2,\sW)
\]
denote the Thom isomorphism.  Then the restriction map
\[
i^*_x:H^2_{0_{\tilde{E}_2}}(\tilde{E}_2,\sW)\to H^2_{0_x}(\tilde{E}_{2x},\sW).
\]
induces an isomorphism
\[
\bar{i}^*_x:H^2_{0_{\tilde{E}_2}}(\tilde{E}_2,\sW)/(s_{0*}(1+\<\bar{q}\>))\to H^2_{0_x}(\tilde{E}_{2x},\sW)
\]
\end{lemma}

\begin{proof} Since $p_1^*Q(x)=-1\mod k^{\times 2}$, it follows that $i_x^*(1+\<\bar{q}\>)=1+\<-1\>=0$ in $W(k)$. Thus 
\[
i_x^*(s_{0*}(1+\<\bar{q}\>))=0\in H^2_{0_x}(\tilde{E}_{2x},\sW),
\]
so we have a well-defined homomorphism
\[
\bar{i}^*_x:H^2_{0_{\tilde{E}_2}}(\tilde{E}_2,\sW)/(s_{0*}(1+\<\bar{q}\>))\to H^2_{0_x}(\tilde{E}_{2x},\sW)
\]
We have the decomposition
\[
H^0(BN, \sW)=W(k)\oplus W(k)\cdot\<\bar{q}\>
\]
and thus the restriction of the Thom isomorphism $s_{0*}:H^0(BN, \sW)\to 
H^2_{0_{\tilde{E}_2}}(\tilde{E}_2,\sW)$ to the factor $W(k)$ induces an isomorphism
\[
\bar{s}_{0*}:W(k)\to H^2_{0_{\tilde{E}_2}}(\tilde{E}_2,\sW)/(s_{0*}(1+\<\bar{q}\>)).
\]
As the composition $\bar{i}^*_x\circ \bar{s}_{0*}$ is just the Thom isomorphism
\[
s_{0x*}:W(k(x))=W(k)\to H^2_{0_x}(\tilde{E}_{2x},\sW)
\]
the result follows.
\end{proof}

\section{Rank two bundles on $BN_T(\SL_2)$}\label{sec:Rank2Bundles}

We pursue an analogy of $N:=N_T(\SL_2)$ with $SO(2)$. We let
\[
\iota:\G_m\to T(\SL_2)
\]
be the isomorphism
\[
\iota(t)=\begin{pmatrix}t&0\\0&t^{-1}\end{pmatrix}
\]
and set
\[
\sigma:=\begin{pmatrix}0&1\\-1&0\end{pmatrix}
\]
 
For $m\ge1$ an integer, define the representations $(F(m),\rho_m)$ of $N$ by  $F(m)=\A^2$ and
\begin{align*}
\rho_m\begin{pmatrix}t&0\\0&t^{-1}\end{pmatrix}&=\begin{pmatrix}t^m&0\\0&t^{-m}\end{pmatrix}\\
\rho_m(\sigma)&=\begin{pmatrix}0&1\\(-1)^m&0\end{pmatrix}
\end{align*}
We let $(\rho_0, F(0))$ denote  trivial  one-dimensional representation.

For $m\ge0$, we let $\rho_m^-:N\to \SL_2$ be the representation  given by  
\begin{align*}
\rho^-_m(\iota(t))&=\rho_m(\iota(t))\\
\rho^-_m(\sigma)&=(-1)\cdot \rho_m(\sigma)
\end{align*}
We let $p^{(m)}:\tilde{O}(m)\to BN$ denote the  rank two vector bundle
\[
\tilde{O}(m)=F(m)\times^NE\SL_2\to N\backslash E\SL_2=BN
\]
corresponding to $\rho_m$, with zero-section $s^{(m)}_0:BN\to \tilde{O}(m)$ and let $\th^{(m)}\in H^2_{0_{\tilde{O}(m)}}(\tilde{O}(m), \sW(\det^{-1}(\tilde{O}(m)))$ denote the Thom class $s^{(m)}_{0*}(1)$. We define the bundle $p^{(m)-}:\tilde{O}^-(m)\to BN$ replacing $\rho_m$ with $\rho_m^-$. We note that $\Pic(BN)\cong \Z/2$, with generator $\gamma:=\tilde{O}^{-}(0)$ and that we have canonical isomorphisms $\det\tilde{O}^\pm(m)\cong O_{BN}$ for $m$ odd and $\det\tilde{O}^\pm(m)\cong \gamma$ for $m>0$ even.

Our goal in this section is to compute the Euler classes $e(\tilde{O}^\pm(m))\in H^2(BN,\sW(\det^{-1}(\tilde{O}^\pm(m))))$.

\begin{lemma}\label{lem:Ident} $\tilde{O}(2)$ is isomorphic to $j^*\sT$ and therefore $e(\tilde{O}(2))=j^*e(\sT)$.
\end{lemma}

\begin{proof} The quotient map $\pi:\P^1\times\P^1\setminus\Delta\to \P^2\setminus C$ is \'etale and we have $\pi^*T_{\P^2\setminus C}\cong T_{\P^1\times\P^1\setminus\Delta}\cong p_1^*T_{\P^1}\oplus p_2^*T_{\P^1}\cong 
p_1^*O_{\P^1}(2)\oplus p_2^*O_{\P^1}(2)$. Under the inclusion $T(\SL_2)\to T(\GL_2)$, the restriction of $\rho_2$ to $T(\SL_2)$ is the restriction of the two-dimensional representation of 
$T(\GL_2)=\G_m\times\G_m$ defining $p_1^*O_{\P^1}(2)\oplus p_2^*O_{\P^1}(2)$, and the action of $\sigma$ on $\pi^*T_{\P^2\setminus C}=p_1^*O_{\P^1}(2)\oplus p_2^*O_{\P^1}(2)$ is via $\rho_2(\sigma)$.
\end{proof}

Let $R$ be a $k$-algebra, essentially of finite type over $k$. We consider an $R$-valued point $x$ of $B_n\SL_2$ coming from an $R$-valued point of $E_n\SL_2$, namely, two elements $v_1, v_2\in \A^{n+2}(R)$ such that the resulting matrix $\begin{pmatrix}v_1\\v_2\end{pmatrix}$ has rank two at each maximal ideal of $R$. This gives us the corresponding $R$-valued point $\bar{x}$ of the Grassmannian $\Gr(2,n)$,
\[
\bar{x}: \Spec R\to \Gr(2,n+2)
\]
together with the isomorphism of $R$-modules $R\cong \det\bar{x}^*E_2$ sending $1\in R$ to $v_1\wedge v_2$. 

We have the $m$th-power map
\[
\psi_{m,a}:F(a)\to F(ma);\ \psi_{m,a}(x,y):=(x^m, y^m),
\]
which is a (non-linear!) $N$-equivariant morphism. This gives us the flat and finite map of $BN$-schemes
\[
\psi_{m,a}:\tilde{O}(a)\to \tilde{O}(ma).
\]

For $V\to X$ a rank $r$ vector bundle with trivialized determinant $\theta:\det V\to O_X$ and 0-section $s_0:X\to V$, and  we write $s_{0*}:\sE^{a,b}(X)\to \sE^{a+2r,b+r}_{0_V}(V;\det^{-1}V)$ for the Thom isomorphism.

\begin{lemma} For $m$ odd, $e(\tilde{O}(m))=s_0^{(1)*}(\psi_{m,1}^*(\th^{(m)}))$ and for $m$ even, 
$e(\tilde{O}(m))=s_0^{(2)*}(\psi_{m/2,2}^*(\th^{(m)}))$.
\end{lemma}

\begin{proof} For $m$ odd, we have $s_0^{(m)}=\psi_{m,1}\circ s_0^{(1)}$, hence
\begin{align*}
e(\tilde{O}(m))&=s_0^{(m)*}(\th^{(m)})\\
&=s_0^{(1)*}(\psi_{m,1}^*(\th^{(m)})).
\end{align*}
The proof for $m$ even is the same.
\end{proof}

Via the isomorphisms
\[
\xymatrix{
W(k)\oplus W(k)\cdot \<\bar{q}\>=H^0(BN,\sW)\ar[r]_-\sim^-{s_{0*}^{(1)}}& H^2_{0_{\tilde{O}(1)}}(\tilde{O}(1),\sW)\\
W(k)\oplus W(k)\cdot \<\bar{q}\>=H^0(BN,\sW)\ar[r]_-\sim^-{s_{0*}^{(2)}}& H^2_{0_{\tilde{O}(2)}}(\tilde{O}(2),\sW(\gamma))
}
\]
we write, for $m$ odd, 
\[
\psi_{m,1}^*(\th^{(m)})=s_{0*}^{(1)}(a_m+b_m\cdot \<\bar{q}\>)\in H^2_{0_{\tilde{O}(1)}}(\tilde{O}(1),\sW),
\]
and for $m$ even
\[
\psi_{m/2,2}^*(\th^{(m)})=s_{0*}^{(2)}(a_m+b_m\cdot \<\bar{q}\>)\in H^2_{0_{\tilde{O}(2)}}(\tilde{O}(2),\sW(\gamma)),
\]
uniquely defining the elements $a_m, b_m\in W(k)$. 

\begin{lemma}\label{lem:reduction1} For $m$ odd, $e(\tilde{O}(m))=(a_m-b_m)p^*e\in H^2(BN, \sW)$ and for $m$ even, $e(\tilde{O}(m))=(a_m-b_m)e(\tilde{O}(2)) \in H^2(BN, \sW(\gamma))$. 
\end{lemma}

\begin{proof} For $m$ odd, the composition
\begin{multline*}
\xymatrix{
W(k)\oplus W(k)\cdot \<\bar{q}\>=H^0(BN,\sW)\ar[r]_-\sim^-{s_{0*}^{(1)}}& H^2_{0_{\tilde{O}(1)}}(\tilde{O}(1),\sW)}\\
\xrightarrow{s_{0}^{(1)*}}
H^2(BN, \sW)=p^*e\cdot W(k)
\end{multline*}
is multiplication with $e(\tilde{O}(1))=p^*e$. Since $(1+\<\bar{q}\>)p^*e=0$, we see that 
$a_m+b_m\cdot \<\bar{q}\>$ gets sent to $(a_m-b_m)p^*e$.

Replacing $e(\tilde{O}(1))$ with $e(\tilde{O}(2))$ and using $\gamma$-twisted cohomology, the proof for $m$ even is the same, using the relation $(1+\<\bar{q}\>)\cdot j^*e(\sT)=0$ and
Lemma~\ref{lem:Ident}.
\end{proof}

Choose a point $x\in BN(k)$ with $p_1^*Q(x)=-1\mod k^{\times 2}$. Write  $\tilde{O}(m)_x$ for the fiber of $\tilde{O}(m)$ at $x$ and let
\[
i_x^{(m)}:\tilde{O}(m)_x\to \tilde{O}(m)
\]
denote the inclusion. Take $m$ odd and let
\[
\psi_{m,1}(x):\tilde{O}(1)_x\to \tilde{O}(m)_x
\]
the restriction of $\psi_{m,1}$.  The map
\[
s^{(m)}_{0x}:x\to \tilde{O}(m)_x
\]
induces an isomorphism
\[
s^{(m)}_{0x*}:W(k)\to H^2_{0_x}(\tilde{O}(m)_x,\sW)
\]
We write $\th^{(m)}_x:=s^{(m)}_{0x*}(1)$.

For $m$ even, we have $\psi_{m/2,2}(x):\tilde{O}(2)_x\to \tilde{O}(m)_x$ and $s^{(m)}_{0x}$ induces an isomorphism
\[
s^{(m)}_{0x*}:W(k)\to H^2_{0_x}(\tilde{O}(m)_x,\sW(\gamma_x)).
\]

\begin{lemma}\label{lem:reduction2} For $m$ odd, $s^{(1)}_{0x*}(a_m-b_m)= \psi_{m,1}(x)^*(\th^{(m)}_x)= \psi_{m,1}(x)^*(s^{(m)}_{0x*}(1))$ and for $m$ even, $s^{(2)}_{0x*}(a_m-b_m)= \psi_{m/2,2}(x)^*(\th^{(m)}_x)= \psi_{m/2.2}(x)^*(s^{(m)}_{0x*}(1))$. 
\end{lemma}

\begin{proof} For $m$ odd, the second identity is just the definition of $\th^{(m)}_x$. We have the commutative diagrams
\[
\xymatrix{
\tilde{O}(1)_x\ar[r]^-{\psi_{m,1}(x)} \ar[d]_{i^{(1)}_x}&\tilde{O}(m)_x \ar[d]^{i^{(m)}_x}\\
\tilde{O}(1)\ar[r]_{\psi_{m,1}}&\tilde{O}(m)
}
\]
and
\[
\xymatrix{
\Spec k\ar[d]_{s_{0x}}\ar[r]^x&BN\ar[d]^{s_0^{(1)}}\\
\tilde{O}(1)_x\ar[r]^{i_x}&\tilde{O}(1)
}
\]
The map $x^*:H^0(BN,\sW)\to W(k)$ is the identity on the summand $W(k)$ of 
$H^0(BN,\sW)$ and is the map $b\cdot\<\bar{q}\>\mapsto -b$ on the summand $W(k)\cdot\<\bar{q}\>$, because $x^*(\<\bar{q}\>)=\<-1\>=-1\in W(k)$. 

Via the respective Thom isomorphisms
\[
s_{0x*}^{(1)}:W(k)\to  H^2_{0_x}(\tilde{O}(1)_x,\sW)
\]
and
\[
s_{0*}^{(1)}:H^0(BN,\sW)\to H^2_{0_{\tilde{O}(1)}}H^2_{0_x}(\tilde{O}(1),\sW)
\]
we see that 
\begin{align*}
 \psi_{m,1}(x)^*(\th^{(m)}_x)&=i_x^{(1)*}(\psi_{m,1}^*(\th^{(m)}))\\
&=i_x^{(1)*}(s_{0*}^{(1)}(a_m+b_m\cdot \<\bar{q}\>))\\
&=s_{0x*}^{(1)}(x^*((a_m+b_m\cdot \<\bar{q}\>)))\\
&=s_{0x*}^{(1)}(a_m-b_m)
\end{align*}

The proof for $m$ even is the same.
\end{proof}

Take $m$ odd. Fixing the point $x\in BN$ as above, we note that the $k$-vector spaces $\tilde{O}(1)_x$ and $\tilde{O}(m)_x$ are both isomorphic to $F=\A^2_k$. We have canonical trivializations $\eta_m:\det \tilde{O}(m)_x\to k$, $\eta_1:\det \tilde{O}(1)_x\to k$. We normalize the  trivializations $\alpha_1:\tilde{O}(1)_x\to \A^2_k$, $\alpha_m:\tilde{O}(m)_x\to \A^2_k$ by requiring that 
$\det\alpha_1\eta_1^{-1}=\det\alpha_m\eta_m^{-1}$ in $k^\times/k^{\times 2}$; let $\lambda\in 
k^\times/k^{\times 2}$ denote this common value.  This gives us the commutative diagram
\[
\xymatrix{
H^2_0(\A^2,\sW)\ar[r]^-{\alpha_m^*}&H^2_{0_x}(\tilde{O}(m)_x,\sW)\\
W(k)\ar[u]^{s_{0*}}\ar[r]_{\times\<\lambda\>}&W(k)\ar[u]_{s_{0_x*}}
}
\]
and similarly for $m=1$. 

Via these trivializations, we may therefore consider $\psi_{m,1}(x)$ as a self-map
\[
\psi_{m,1}(x):\A^2_k\to \A^2_k
\]
inducing the map on the Thom spaces $\A^2/(\A^2\setminus\{0\})$
\[
\th(\psi_{m,1}(x)):\A^2/(\A^2\setminus\{0\})\to \A^2/(\A^2\setminus\{0\}).
\]
Passing to $\SH(k)$, $\th(\psi_{m,1}(x))$ gives an endomorphism of the sphere spectrum, 
\[
\Sigma^\infty_T\th(\psi_m(x))\in \End_{\SH(k)}(\mS_k)=\GW(k).
\]
Let $\gamma_m\in W(k)$ be the image of $\Sigma^\infty_T\th(\psi_m(x))$ in $W(k)$, via Morel's isomorphism $\End_{\SH(k)}(\mS_k)\cong \GW(k)$ and the surjection $\GW(k)\to W(k)$.  It follows from lemmas~\ref{lem:reduction1}, \ref{lem:reduction2} that
\begin{equation}\label{eqn:mainOdd}
e(\tilde{O}(m))=\gamma_m\cdot p_1^*e.
\end{equation}

For $m$ even, we have a canonical isomorphism $\eta_m:\det \tilde{O}(m)_x\to \gamma_x$. We take trivializations   $\alpha_2:\tilde{O}(2)_x\to \A^2_k$, $\alpha_m:\tilde{O}(m)_x\to \A^2_k$ such that the diagram
\begin{equation}\label{eqn:Triv}
\xymatrix{
\det \tilde{O}(2)_x\ar[rr]^{\det\alpha_m^{-1}\circ \det \alpha_2}\ar[dr]_{\eta_2}&&\det \tilde{O}(m)_x
\ar[dl]^{\eta_m}\\
&\gamma_x
}
\end{equation}
commutes up to multiplication by an element in $k^{\times 2}$.  Let $\gamma_m\in W(k)$ be the image of $\Sigma^\infty_T\th(\psi_{m/2,2}(x))$ in $W(k)$. Using the fact that $e(\tilde{O}(2))$ is a $W(k)$-generator for $H^2(BN, \sW(\gamma))$ (Proposition~\ref{prop:WittCoh1}  and Lemma~\ref{lem:Ident}), it follows as for odd $m$  that
\begin{equation}\label{eqn:mainEven}
e(\tilde{O}(m))=\gamma_m\cdot e(\tilde{O}(2)).
\end{equation}

For $m$ odd, we can just as well compute $\gamma_m$  as the endomorphism of $\mS_k$ induced by the $\P^1$-desuspension of $\th(\psi_{m,1}(x))$  (noting that $\psi_m(x)$ is homogeneous)
\[
\P(\psi_{m,1}(x)):\P^1_k\to \P^1_k
\]
For $m$ even, we do the same, using $\P(\psi_{m/2,2}(x)):\P^1_k\to \P^1_k$ to compute $\gamma_m$ 

We summarize the discussion as follows.
\begin{lemma}\label{lem:Upshot}  Let  $x$ be a $k$-point of $BN$ such that $x^*(\<\bar{q}\>)=\<-1\>=-1\in W(k)$. For $m$ odd,  suppose we have chosen isomorphisms  $\alpha_1:\tilde{O}(1)_x\to \A^2_k$, $\alpha_m:\tilde{O}(m)_x\to \A^2_k$ such $\det\alpha_1\eta_1^{-1}=\det\alpha_m\eta_m^{-1}$ in $k^\times/k^{\times 2}$.  Let $\gamma_m\in W(k)$ be the image of $[\P(\psi_{m,1}(x))]\in \End_{\SH(k)}(\mS_k)=\GW(k)$. Then 
$e(\tilde{O}(m))=\gamma_m\cdot p_1^*e$ in $H^2(BN, \sW)$. 

For $m$ even,  suppose we have chosen isomorphisms  $\alpha_2:\tilde{O}(2)_x\to \A^2_k$, $\alpha_m:\tilde{O}(m)_x\to \A^2_k$ such that the diagram \eqref{eqn:Triv} commutes up to multiplication by an element in $k^{\times 2}$.   Let $\gamma_m\in W(k)$ be the image of $[\P(\psi_{m/2,2}(x))]\in \End_{\SH(k)}(\mS_k)=\GW(k)$. Then 
$e(\tilde{O}(m))=\gamma_m\cdot e(\tilde{O}(2))$ in $H^2(BN, \sW(\gamma))$.
\end{lemma}

\section{Descent, normalized bases and the computation of $\gamma_m$}\label{sec:Descent}

We will find a suitable $x$ and compatible isomorphisms $\alpha_1$, $\alpha_2$, $\alpha_m$ by using a descent description of the bundles $\tilde{O}(m)$.

We have $B_0N:=N\backslash \GL_2\subset  BN$,
a $\G_m$-bundle over $N_T\backslash \GL_2=N\backslash \SL_2=\P^2\setminus C$. With respect to the projection $BN\to \BGL_2$, $B_0N$ is the fiber over the base-point $0=\GL_2\backslash \GL_2\in \BGL_2$.

We describe the bundle $\tilde{O}(m)$ restricted to $B_0N$ in terms of descent data. Let $T(\GL_2)\subset N_T(\GL_2)$ be the standard torus in $\GL_2$. The representation $\rho_m$ extends to the representation
\[
\bar{\rho}_m:N_T(\GL_2)\to \GL_2
\]
with $\bar{\rho}_m(\sigma)=\rho_m(\sigma)$ and
\[
\bar{\rho}_m\begin{pmatrix}t_1&0\\0&t_2\end{pmatrix}=\begin{pmatrix}t_1^m&0\\0&t_2^m\end{pmatrix}
\]
Thus $\tilde{O}(m)$ restricted to $B_0N$ is the pullback  of the bundle $\bar{O}(m)$ on $N_T\backslash \GL_2=\P^2\setminus C$ defined by the representation $\bar{\rho}_m$, via the quotient map  $N\backslash \GL_2\to N_T\backslash \GL_2$. 

We let $k(i):=k[T]/(T^2+1)$, giving us the separable $k$-algebra with automorphism $c:k(i)\to k(i)$, $c(i)=-i$, over $k$. For $z=a+bi\in k(i)$, $a,b\in k$, we write $a=\Re z$, $b=\Im z$ and extend this notation in the obvious way to the polynomial ring $k(i)[X_0,\ldots, X_r]$.

We will take a pair of $k(i)$-points $y_1, y_2\in \A^2(k(i))$, linearly independent over $k(i)$, such that
\[
\sigma\cdot \begin{pmatrix}y_1,y_2\end{pmatrix}=
\begin{pmatrix}-i&0\\0&i\end{pmatrix}
\begin{pmatrix}y_1,y_2\end{pmatrix}^c
\]
 This gives us the $k$-point $x\in N\backslash \GL_2=B_0N$. Then a $k(i)$-basis $v_{1, m}, v_{2,m}$ of $F(m)\otimes k(i)=k(i)^2$ descends to a $k$-basis of $\tilde{O}(m)_x$ if the equation
\begin{equation}\label{eqn:descent}
\begin{pmatrix}v_{1m}\\v_{2,m}\end{pmatrix}\cdot \rho_m(\sigma)=\begin{pmatrix}v_{1m}\\v_{2,m}\end{pmatrix}^c\cdot \rho_m\begin{pmatrix}-i&0\\0&i\end{pmatrix}
\end{equation}
is satisfied.

To complete the computation of $e(\tilde{O}(m))$, we have
\begin{theorem}\label{thm:EulerClassOm} Suppose $k$ is a field of characteristic zero, or of characteristic $\ell>2$, with $(\ell,m)=1$.  For $m$ odd, define $\epsilon(m)=+1$ for $m\equiv 1\mod 4$, $\epsilon(m)=-1$ for $m\equiv 3\mod 4$.  Then 
\[
e(\tilde{O}(m))=\begin{cases}  \epsilon(m)\cdot m\cdot p^*e\in H^2(BN,\sW)&\text{ for }m\text{ odd},\\
\frac{m}{2}\cdot\tilde{e}\in H^2(BN,\sW(\gamma))&\text{ for }m\equiv 2\mod 4\\
-\frac{m}{2}\cdot\tilde{e}\in H^2(BN,\sW(\gamma))&\text{ for }m\equiv 0\mod 4
\end{cases}
\]
Moreover $\tilde{e}^2=4p^*e^2$ in $H^4(BN,\sW)$ and $e(\tilde{O}^-(m))=-e(\tilde{O}(m))$.
\end{theorem}

\begin{proof} Noting that $(x,y)\mapsto (-x, y)$ defines an isomorphism of $\rho_m$ with $\rho_m^-$ (as right representations) and induces $(-1)$ on the determinant, the assertion for 
$e(\tilde{O}^-(m))$ follows.

The proof relies on the approach set forth in  Lemma~\ref{lem:Upshot}. We find a suitable $k$-point $x$ of $B_0N$ and compatible trivializations $\alpha_m$, compute the maps $\P(\psi_{m,a})$ and use these to compute the factor $\gamma_m$. 

We begin by finding a $k$-point $x\in B_0N=N\backslash\GL_2$ lying over 
 the point $[1:0:1]\in N_T(\GL_2)\backslash\GL_2=\P^2\setminus C$  with $Q(x)=-1\mod k^{\times 2}$.  

We have the double cover $T(\SL_2)\backslash\SL_2\to N\backslash\SL_2$, that is, 
$\P^1\times\P^1\setminus\Delta\to \P^2\setminus C$, and the pair of conjugate points $y,\bar{y}\in \P^1\times\P^1\setminus\Delta$ lying over $[1:0:1]$. As the map $\pi:\P^1\times\P^1\to \P^2$ is given by
\[
\pi([x_0:x_1], [y_0, y_1])=[x_0y_0: x_0y_1+x_1y_0:x_1y_1]
\]
we compute the conjugate points $y, \bar{y}$ as
\[
y=[1:-i], \bar{y}=[1:i].
\]
We lift $y,\bar{y}$ to the pair $y_1=(1, -i)$, $y_2=(-i, 1)$ $y_1, y_2\in k(i)^2$, giving the matrix
\[
x^*:=\begin{pmatrix}y_1\\y_2\end{pmatrix}=\begin{pmatrix}1&-i\\-i&1\end{pmatrix}\in \GL_2(k(i))
\]
with determinant 2. The pair $y_1, y_2$ defines a $k$-point $x$ of $B_0N=N\backslash \GL_2$ lying over $(1:0:1)\in N_T\backslash\GL_2$, since
\[
\sigma\cdot \begin{pmatrix}1&-i\\-i&1\end{pmatrix} =
\begin{pmatrix}-i&0\\0&i\end{pmatrix}\cdot \begin{pmatrix}1&-i\\-i&1\end{pmatrix}^c.
\] 

Identifying the fiber $\tilde{O}(m)_{x^*}$ with $F(m)\times x^*=k(i)^2$, we have the standard basis $e_{1}=(1,0)$, $e_{2}=(0,1)$ for $\tilde{O}(m)_{x^*}$. We first consider the case of odd $m$. 

For $m\equiv1\mod 4$, define $v_{1,m}=(1, i)$, $v_{2,m}=(i,1)$. We have
\[
\begin{pmatrix}v_{1, m}\\v_{2,m}\end{pmatrix}\rho_m(\sigma)=
 \begin{pmatrix}v_{1, m}\\v_{2,m}\end{pmatrix}^c\rho_m\begin{pmatrix}-i&0\\0&i\end{pmatrix}
\]
so $v_{1,m}, v_{2,m}$ satisfy the descent equation \eqref{eqn:descent} and thus descend to a $k$-basis of $\tilde{O}(m)_x$ with determinant $2\cdot e_1\wedge e_2$. 

For $m\equiv 3\mod 4$, define $v_{1,m}=(1,-i)$, $v_{2,m}=(-i,1)$, and we have as above
\[
\begin{pmatrix}v_{1, m}\\v_{2,m}\end{pmatrix}\rho_m(\sigma)=
 \begin{pmatrix}v_{1, m}\\v_{2,m}\end{pmatrix}^c\rho_m\begin{pmatrix}-i&0\\0&i\end{pmatrix}.
\]
The $k$-basis $v_{1,m}, v_{2,m}$ also has determinant $2\cdot e_1\wedge e_2$.

We remark that in all cases of odd $m$, our basis has determinant   $2\cdot e_1\wedge e_2$, in particular
\[
\left<\frac{v_{1,m}\wedge v_{2,m}}{v_{1,1}\wedge v_{2,1}}\right>=1\text{ in }W(k).
\]

For $m\equiv 2\mod 4$, define  $v_{1,m}=(1,-1)$, $v_{2,m}=(i,i)$.  Then
\[
\begin{pmatrix}v_{1, m}\\v_{2,m}\end{pmatrix}\rho_m(\sigma)=\begin{pmatrix}v_{1, m}\\v_{2,m}\end{pmatrix}^c \cdot \rho_m\begin{pmatrix}-i&0\\0&i\end{pmatrix}
\]
The basis $v_{1,m}, v_{2,m}$ has determinant $2i\cdot e_1\wedge e_2$, but in the case of even $m$, the representation $\rho_m$ does not land in $\SL_2$, and thus the determinant bundle $\det\tilde{O}(2)$ is the 2-torsion generator $\gamma$ of $\Pic(N\backslash\GL_2)$, corresponding to the character $\rho(0)^-$ of $N$. In any case,
$v_{1,m}\wedge v_{2,m}$  gives a generator for $\det\tilde{O}(2)_x$.

For $m\equiv 0\mod 4$, we use the basis $v_{1,m}=(1,1)$, $v_{2,m}=(-i,i)$ for $\tilde{O}(m)_x$, which  satisfies the  descent equation
\[
\begin{pmatrix}v_{1, m}\\v_{2,m}\end{pmatrix}\rho_m(\sigma)=\begin{pmatrix}v_{1, m}\\v_{2,m}\end{pmatrix}^c \cdot \rho_m\begin{pmatrix}-i&0\\0&i\end{pmatrix}
\] 
and has the determinant $2i\cdot e_1\wedge e_2$. 

For all $m$ even, we have
\[
\left<\frac{v_{1,m}\wedge v_{2,m}}{v_{1,2}\wedge v_{2,2}}\right>=1\text{ in }W(k).
\]

The multiplication map $\psi_{m,a}:\tilde{O}(a)_{\tilde{x}}\to \tilde{O}(am)_{\tilde{x}}$ is given by 
\[
\psi_m(x_0\cdot e_1+x_1 e_2)=x_0^m\cdot e_1+x_1^m\cdot e_2.
\]
We express $\psi_m$ with respect our bases for $\tilde{O}(a)_x$ and $\tilde{O}(am)_x$. 

For $a, m\equiv 1\mod 4$, we have
\[
e_1=\frac{1}{2}(v_{1,a}-i v_{2,a})=\frac{1}{2}(v_{1,am}-i v_{2,am}),\ e_2=-\frac{i}{2}(v_{1,a}+ i v_{2,a})=-\frac{i}{2}(v_{1,am}+ i v_{2,am})
\]
and $\psi_{m,a}$ is given by
\begin{multline*}
\psi_{m,a}(x_0\cdot v_{1,a}+x_1\cdot v_{2,a})=
\psi_{m,a}((x_0+ix_1)\cdot e_1+i (x_0-ix_1)\cdot e_2)\\=
(x_0+ix_1)^me_1+i (x_0-ix_1)^me_2\\=
\frac{1}{2}((x_0+ix_1)^m+(x_0-ix_1)^m)v_{1,am}-\frac{i}{2}((x_0+ix_1)^m-(x_0-ix_i)^m)v_{2,am}\\
=\Re(x_0+ix_1)^m\cdot v_{1,m}+\Im(x_0+ix_1)^m\cdot v_{2,am}.
\end{multline*}

For $m\equiv 3\mod 4$, $a\equiv 1\mod 4$, we have
\[
e_1=\frac{1}{2}(v_{1,am}+i v_{2,am}),\ e_2=\frac{i}{2}(v_{1,am}-i v_{2,am})
\]
and $\psi_{m,a}$ is given by
\begin{multline*}
\psi_{m,a}(x_0\cdot v_{1,a}+x_1\cdot v_{2,a})=
\psi_{m,a}((x_0+ix_1)\cdot e_1+i (x_0-ix_1)\cdot e_2)\\=
(x_0+ix_1)^me_1-i (x_0-ix_1)^me_2\\=
\frac{1}{2}((x_0+ix_1)^m+(x_0-ix_1)^m)v_{1,am}+\frac{i}{2}((x_0+ix_1)^m-(x_0-ix_i)^m)v_{2,am}\\
=\Re(x_0-ix_1)^m\cdot v_{1,am}+\Im(x_0-ix_1)^m\cdot v_{2,am}.
\end{multline*}
For $a\equiv 3\mod 4$, we arrive at the same formulas, giving in the end the formulas for all odd $a$
\begin{multline*}
\psi_{m,a}(x_0\cdot v_{1,a}+x_1\cdot v_{2,a})\\=\begin{cases}
\Re(x_0+ix_1)^m\cdot v_{1, am}+\Im(x_0+ix_1)^m\cdot v_{2,am}&\text{ for }m\equiv 1\mod 4,\\
\Re(x_0-ix_1)^m\cdot v_{1,am}+\Im(x_0-ix_1)^m\cdot v_{2,am}&\text{ for }m\equiv 3\mod 4
\end{cases}
\end{multline*}

For even $m$, we write $m=2^rn$ with $n$ odd and write $\psi_{m/2,2a}:\tilde{O}(2a)\to \tilde{O}(am)$ as the composition
\[
\tilde{O}(2a)\xrightarrow{\psi_{n,2a}}\tilde{O}(2an)\xrightarrow{\psi_{2, 2an}}
\tilde{O}(2^2an)\xrightarrow{\psi_{2, 2^2an}}\ldots
\xrightarrow{\psi_{2, 2^{r-1}an}} \tilde{O}(am)
\]
A computation as above yields
\[
\psi_{n, 2a}(x_0\cdot v_{1,2a}+x_1\cdot v_{2,2a})=\Re (x_0+ix_1)^n\cdot v_{1,2an}+\Im (x_0+ix_1)^n\cdot v_{2,2an}
\]
For $\psi_{2, 2^san}$ there are two cases: We take $a$ odd. Then
\[
\psi_{2,2an}(x_0\cdot v_{1,2an}+x_1\cdot v_{2,2an})=\Re (x_0-ix_1)^2\cdot v_{1, 4an}+\Im(x_0-ix_1)^2\cdot v_{2,4an}
\]
and for all $s\ge 2$ 
\[
\psi_{2,2^san}(x_0\cdot v_{1,2^san}+x_1\cdot v_{2,2^san})=\Re (x_0+ix_1)^2\cdot v_{1, 2^{s+1}an}+\Im(x_0+ix_1)^2\cdot v_{2,2^{s+1}an}
\]
This second formula also takes care of the case of even $a$. 

Finally, we need to consider the map $\psi_{2,1}:\tilde{O}(1)\to \tilde{O}(2)$. We have
\[
e_1=\frac{1}{2}(v_{1,2}-i v_{2,2}),\ e_2=-\frac{1}{2}(v_{1,2}+i v_{2,2})
\]
and $\psi_{2,1}$ is given by
\begin{multline*}
\psi_{2,1}(x_0\cdot v_{1,1}+x_1\cdot v_{2,1})=
\psi_{2,1}((x_0+ix_1)\cdot e_1+i(x_0-ix_1)\cdot e_2)\\=
(x_0+ix_1)^2\cdot e_1-(x_0-ix_1)^2\cdot e_2\\=
\frac{1}{2}((x_0+ix_1)^2+(x_0+ix_1)^2)v_{1,m}-\frac{i}{2}((x_0+ix_1)^2-(x_0-ix_i)^2)v_{2,2}\\
=\Re(x_0+ix_1)^2\cdot v_{1,2}+\Im(x_0+ix_1)^2\cdot v_{2,2}.
\end{multline*}

Thus, as an endomorphism of $\P^1$, we have the expression for $\P(\psi_{m,a})$:
\begin{multline*}
\P(\psi_{m,a})([x_0:x_1])\\=\begin{cases}[\Re(x_0+ix_1)^m:\Im(x_0+ix_1)^m]&\text{ for }m\equiv 1\mod 4, a\text{ odd},\\
&\text{ for }m\text{ odd, } a\text{ even},\\
&\text{ for }m=2, a\equiv0\mod 4,\\
&\text{ and for }m=2, a=1,\\
[\Re(x_0-ix_1)^m:\Im(x_0-ix_1)^m]&\text{ for }m\equiv 3\mod 4, a\text{ odd},\\
&\text{ and for }m=2, a\equiv 2\mod 4.
\end{cases}
\end{multline*}
In all cases except for $\psi_{2,1}:\tilde{O}(1)\to \tilde{O}(2)$, the ``relative determinants'' are 1 in $W(k)$.

To complete the computation, we  use Lemma~\ref{lem:Degree} below to determine the $\A^1$-degree $[\P(\psi_{m,a})]\in W(k)$ of $\P(\psi_{m,a}):\P^1_\Q\to \P^1_\Q$ in cases described above. Combined with our table above, that lemma gives, for $p$ a prime, 
\[
[\P(\psi_{p,a})]=\begin{cases} \epsilon(p)\cdot p\<1\>&\text{ for }p\text{ odd and }a\text{ odd},\\
p\<1\>&\text{ for }p\text{ odd and }a\text{ even},\\
2\<1\>&\text{ for } p=2, a\equiv 0\mod 4,\\
-2\<1\>&\text{ for }p=2, a\equiv 2 \mod 4\\
2\<1\>&\text{ for } p=2, a=1.
\end{cases}
\]

This computation of the $\A^1$-degree of $\P(\psi_{p,a})$ together with Lemma~\ref{lem:Upshot} yields the computation of $e(\tilde{O}(m))$ as asserted in Theorem~\ref{thm:EulerClassOm}: in the case of odd $m$, we write $m=\prod_{i=1}^rp_i$ with the $p_i$ odd primes, let $a_i=\prod_{j=1}^{i-1}p_i$ and write $\psi_{m,1}=\psi_{p_r, a_r}\circ\ldots\circ \psi_{p_1,1}$.  This gives $e(\tilde{O}(m))=\prod_{i=1}^r\epsilon(p_i)\cdot p_i\cdot p^*e=\epsilon(m)\cdot m\cdot p^*e$.  For $m=2^r\cdot n$ with $n$ odd, we write
$\psi_{m/2,2}=\psi_{2,m/2}\circ\ldots\circ\psi_{2, 2n}\circ\psi_{n,2}$, which, computing $[\P(\psi_{n,2})]$ as above,  gives 
$e(\tilde{O}(m))=\pm 2^{r-1}n\cdot e(\tilde{O}(2))$, with the sign $+$ for $r=1$, $-$ for $r>1$. 

We complete the proof by showing $e(\tilde{O}(2))^2=4e(\tilde{O}(1))^2$ in $H^4(BN,\sW)$. The bundle $\tilde{O}(2)\oplus\tilde{O}(2)$ is the bundle associated to the representation 
\[
\rho_2\oplus\rho_2:N\to \SL_4
\]
As $\det^{-2}(\tilde{O}(2))$ is the line bundle associated to the representation $(\rho_0^-)^2=\rho_0$, we have a canonical isomorphism $\det^{-2}(\tilde{O}(2))\cong O_{BN}$, and via the induced identification of $H^4(BN, \sW(\det^{-2}(\tilde{O}(2)))\cong H^4(BN, \sW)$, we have  $e(\tilde{O}(2))^2=e(\tilde{O}(2)\oplus\tilde{O}(2))$. We have the multiplication map 
\[
\psi_2\times\psi_2:\tilde{O}(1)\oplus\tilde{O}(1)\to
\tilde{O}(2)\oplus\tilde{O}(2)
\]
Using the bases $(v_{1,1},0)$, $(v_{2,1},0)$, $(0,v_{1,1})$, $(0, v_{2,1})$ for 
$(\tilde{O}(1)\oplus\tilde{O}(1))_x$ and $(v_{1,2},0)$, $(v_{2,2},0)$, $(0,v_{1,2})$, $(0, v_{2,2})$ for 
$(\tilde{O}(2)\oplus\tilde{O}(2))_x$, we have the respective determinants $4$ and $-4$ times the generator $(e_1,0)\wedge(e_2,0)\wedge(0,e_1)\wedge (0,e_2)$, that is, a relative determinant of -1. Note  that $e(\tilde{O}(2))$ lives in $H^2(BN,\sW(\det^{-1}\tilde{O}(2)))$, and to put   $e(\tilde{O}(2))^2\in H^4(BN,\sW(\det^{-2}\tilde{O}(2)))$ into $H^4(BN,\sW)$ we need to multiply with the square of a local generator for $\det\tilde{O}(2)$ to give the isomorphism $\sW(\det^{-2}\tilde{O}(2))\to \sW$. At the point $x$, a local generator is $v_{1,2}\wedge v_{2,2}=2i e_1\wedge e_2$, so multiplying by $(v_{1,2}\wedge v_{2,2})^2$ introduces the additional factor $\<-4\>=-1$ into the computation. Putting this all together, we see  that the $\A^1$ degree of 
\[
\P(\psi_{2,1})\wedge\P(\psi_{2,1}):\P^1\wedge\P^1\to \P^1\wedge\P^1
\]
computes $e(\tilde{O}(2))^2$. As $\P(\psi_{2,1})$ has $\A^1$-degree $2\<1\>$, this gives $e(\tilde{O}(2))^2=4p^*e^2$.
\end{proof} 

We finish this section with proofs of results we have used in the proof of Theorem~\ref{thm:EulerClassOm}.

\begin{lemma}\label{lem:Degree} Let $k$ be a perfect field and let $m$ be an integer.  Consider the maps $G_{m\pm}:\P^1_k\to \P^1_k$
\begin{enumerate}
\item[(a)]  $G_{m+}(x_0:x_1)=(\Re(x_0+ix_1)^m:\Im(x_0+ix_1)^m)$
\item[(b)] $G_{m-}(x_0:x_1)=(\Re(x_0-ix_1)^m:\Im(x_0-ix_1)^m)$,
\end{enumerate}
and let  $[G_{m\pm}]\in W(k)$ be the image of the class of $G_{\pm m}$ in $\End_{\SH(k)}(\mS_k)=\GW(k)$ under the surjection $\GW(k)\to W(k)$.  Let $p$ be a prime. If   $\Char k=\ell>0$ we suppose that  $(\ell, 2p)=1$. Then  $[G_{p\pm}]=\pm p\cdot \<1\>$.
\end{lemma}

\begin{proof} We first consider the case of characteristic zero; we may assume $k=\Q$.
We compute the $\A^1$-degree of $g:=G_{p\pm}:\P^1_k\to \P^1_k$ as the Brouwer degree (see \cite[Chapter 2]{MorelNewton}): Take $\A^1\subset \P^1$ the open subscheme $x_1\neq0$ and use the parameter $t=x_0/x_1$ on $\A^1$. We pick a regular value $s\in \A^1(\Q)\subset\P^1(\Q)$ with $g^{-1}(s)\subset \A^1$, and take the sum of the trace forms $Tr_{\Q(t_i)/\Q}$ of the 1-dimensional quadratic forms $\<\frac{dg}{dt}(t_i)\>\in \GW(k(t_i))$ at the closed point $t_i$ of inverse image $g^{-1}(s)=\{t_1,\ldots, t_r\}$. 

We first consider the case $g=G_{p+}$ with $p$ odd. We take  $s=(0:1)$. $g^{-1}(s)$ consists of two closed points: $t_0=(0:1)$ and $t_1\in \{x_0\neq0\}\subset \P^1$. Indeed $g^{-1}(0:1)$ is the set of solutions to $\Im(a+ib)^p=0$, up to scaling. One solution is $(a:b)=(0:1)$, since $p$ is odd. For $a\neq0$, we look for solutions in $\C$ (up to scaling by $\R$) to 
\[
(a+ib)^p=\lambda\cdot i;\quad \lambda\in \R
\]
which gives 
\[
(a+ib)^{4p}=\lambda^4=\mu \in \R_{>0};\ a\neq0.
\]
Dividing by $\mu^{1/4p}$, we may assume that $(a+ib)^{4p}=1$, in other words, $a+ib$ is a $4p$th root of unity. The real number $a/b$ is thus a number of the form $\cot(m\pi/2p)$, and the set of $p-1$ solutions $(a:b)\in \P^1(\C)$ with $a\neq0$ is exactly the set 
\[
S:=\{(a:b)=(x:1)\mid x=\cot(m\pi/2p), m=1,\ldots, p-1\}.
\]
As 
\[
\cot(m\pi/2p)=\frac{i\cdot(e^{2m\pi i/4p}+e^{-2m\pi i/4p})}{e^{2m\pi i/4p}-e^{-2m\pi i/4p}}
\]
is in the real subfield $F_p\subset \Q(\zeta_{4p})$, which has degree $p-1$ over $\Q$, and $S$ consists of the $p-1$ conjugates of $\cot(\pi/2p)$ over $\Q$, it follows that the closed point $t_1=(a:1)$ of $\P^1(\Q)$ corresponding to $(\cot(\pi/2p):1)\in \P^1(\bar{\Q})$ has residue field $F_p$.  

We now compute the derivative $g'(t)=dg/dt$ of $g$ at $t_0$ and $t_1$, first for $p\equiv1\mod 4$.

We use the parameter $t=x_0/x_1$, so
\[
g(t)=\frac{\Re(t+i)^p}{\Im(t+i)^p}.  
\]
At $t_0=(0:1)$, we find
\[
g(t)=p\cdot t+\text{higher order terms}
\]
so $g'(0)=p$, which gives us the term $\<p\>$.

We now consider $t_1=(a_1: 1)$ Since $g(t)$ is a rational function in one variable $t$, to make the computation of $g'(t)$ we may consider $g$ as a real-valued function of a real variable $t$. We change coordinates: $x_0=r\cos\theta$, $x_1=r\sin\theta$,  $t=\cot\theta$, and
\[
g(\theta)=\frac{\Re(r^p(\cos\theta+i\sin\theta)^p)}{\Im(r^p(\cos\theta+i\sin\theta)^p)}=
\frac{\cos p\theta}{\sin p\theta}=\cot p\theta.
\]
Thus
\begin{align*}
\frac{dg}{dt}&=\left(\frac{d \cot p\theta}{d\theta}\right)\cdot\left(\frac{dt}{d\theta}\right)^{-1}\\
&=p\cdot \frac{-1}{\sin^2 p\theta}\cdot (-\sin^2\theta)\\
&=p\cdot \frac{\sin^2\theta}{\sin^2 p\theta}\\
&=\frac{p\cdot (1+t^2)^{p-1}}{h(t)^2}
\end{align*}
with
\[
h(t)= \sum_{j=0}^{\frac{p-1}{2}}(-1)^j{\small\begin{pmatrix}p\\p-2j-1\end{pmatrix}}t^{p-2j-1}.
\]
As $\Q(\sqrt{(-1)^{p-1/2}p})\subset \Q(\zeta_p)$, it follows that $\Q(\sqrt{p})\subset F_p$, so $p$ is a square in $F_p$ and thus $g'(t_1)$ 
is a square in $F_p$. This gives the $\A^1$ Brouwer degree of $g$ as
\[
\<p\>+\Tr_{F_p/\Q}(\<1\>)
\]
which by Lemma~\ref{lem:Trace} below is equivalent to a sum of $p$ squares, so $[G_{p+}]=p\cdot\<1\>\in W(k)$.

Next, we take $g=G_{p-}$, $p$ still odd. We have
\[
g(t)=\frac{\Re(t-i)^p}{\Im(t-i)^p},
\]
which is exactly the negative of the case $G_{p+}$, so $g'(0)=-p$ and $g'(t_1)=\frac{-p(1+a_1)^{p-1}}{h(a_1)^2}$. This gives
\[
[G_{p-}]=
\<-p\>+\Tr_{F_p/\Q}(\<-1\>)=-p\cdot\<1\>\in W(\Q).
\]

For $p=2$, we have 
\[
G_{2+}(x_0:x_1)=(x_0^2-x_1^2, 2x_0x_1).
\]
With $t=x_0/x_1$, this gives the map 
\[
g(t)=\frac{1}{2}\cdot(t-t^{-1});\quad g'(t)=\frac{1}{2}\cdot(1+t^{-2})
\]
so $g^{-1}(0)=\{1, -1\}$,  $g'(1)=g'(-1)=1$ and the $\A^1$ Brouwer degree is $\<1\>+\<1\>=2\<1\>$. 

Similarly
\[
G_{2-}(x_0:x_1)=(x_0^2-x_1^2, -2x_0x_1).
\]
With $t=x_0/x_1$, this is the map 
\[
g(t)=-\frac{1}{2}\cdot(t-t^{-1});\quad g'(t)=-\frac{1}{2}\cdot(1+t^{-2})
\]
and the $\A^1$ Brouwer degree is $\<-1\>+\<-1\>=-2\<1\>\in W(k)$.

We now consider $k$ of characteristic $\ell>0$ with $(\ell, 2p)=1$.  We take  $k=\F_\ell$,   $\ell\neq 2, p$ and we use the same argument as above, with the following modifications, first for odd $p$:\\[10pt]
i. For $F\in \{\Q(\zeta_{4p}), \Q(\zeta_p), F_p, \Q(i), \Q(\sqrt{p})\}$, the number ring $\sO_F$ is unramified at $\ell$. In the above discussion, we replace $F$ with the \'etale algebra $\sO_F\otimes_\Z\F_\ell$.  \\[5pt]
ii. We note that $\cot(\pi/2p)\in F_p$ is actually an algebraic integer; we replace  $\cot(\pi/2p)\in \sO_{F_p}$ with its image in $\sO_{F_p}\otimes_\Z\F_\ell$. This shows that the closures of $0, a_1$ in $\Spec \Z[t]$ are finite over $\Spec \Z$; we  let $\bar{0}, \bar{a}_1\in \Spec \F_\ell[t]$ denote their respective specializations.\\[5pt]
iii. We replace the polynomials $\Re(x_0+ix_1)^p, \Im(x_0+ix_1)^p\in \Z[x_0, x_1]$ with their respective images in $\F_\ell[x_0, x_1]$. Letting $\mathfrak{m}_{0,a_1}\subset \Z_{(\ell)}[t]$ be the ideal of $\{0, a_1\}$, we similarly replace $h(t)\in \Z[t]$ and  $g(t)\in \Z_{(\ell)}[t]_{(\mathfrak{m}_{0,a_1})}$ with $\bar{h}(t)\in \F_\ell[t]$ and $\bar{g}(t)\in \F_\ell[t]_{(\bar{\mathfrak{m}}_{\bar{0},\bar{a}_1})}$. \\[5pt]
iv. We replace the trace form $Q_{4p}$ for $F_p$ over $\Q$ with the trace form $\bar{Q}_{4p}$  for the \'etale algebra $\sO_{F_p}\otimes_\Z\F_\ell$ over $\F_\ell$. Noting again that $\sO_{F_p}$ is unramified at $q$, it follows from Lemma~\ref{lem:Trace} and a specialization argument that $\<p\>+\bar{Q}_{4p}=\sum_{i=1}^px_i^2$ in $\GW(\F_\ell)$.\\[10pt]
With these changes and comments, the argument for characteristic zero goes through in characteristic $\ell$.

For $p=2$, $\ell\neq2$, the direct computation given for $\P(\psi_2)$ goes through over $\F_\ell$ without change. 
\end{proof}

\begin{lemma} \label{lem:Trace} Let $p$ be an odd prime, $\zeta_{4p}$ a primitive $4p$th root of unity and let $F_p$ be the real subfield $\Q(\zeta_{4p}+\zeta_{4p}^{-1})$ of $\Q(\zeta_{4p})$. Let  $Q_{4p}$ denote the trace form for the field extension $F_p/\Q$. Then the quadratic form $T_p:=\<p\>+Q_{4p}$ is equivalent to the sum of squares $\sum_{i=1}^px_i^2$.
\end{lemma}

\begin{proof} We have the lattice $(\Z^n,\cdot)$, with $\cdot$ the usual dot product.  For $n\ge2$ an integer, we have the sublattice $A_{n-1}$ of $(\Z^n,\cdot)$, given as the hyperplane 
\[
A_{n-1}:=\{(x_1,\ldots, x_n)\in \Z^n\mid \sum_{i=1}^nx_i=0\}; 
\]
this gives us the integral quadratic form $q_{A_{n-1}}$. Bayer and Suarez \cite[pg. 222]{BayerSuarez} have shown that $q_{A_{p-1}}$ is isometric to an integral model for $Q_{4p}$. Thus  $Q_{4p}+\<p\>$ is isometric to the quadratic form associated to the full sublattice $A_{p-1}+\Z\cdot(1,\ldots, 1)$ of $(\Z^p,\cdot)$. As
\[
\Q\cdot (A_{p-1}+\Z\cdot(1,\ldots, 1))=\Q^p,
\]
it follows that $Q_{4p}+\<p\>$ is equivalent (as a quadratic form over $\Q$) to $p\<1\>$.  

Another proof uses  Serre's theorem \cite[Th\'eor\`eme 1]{Serre}  on the Hasse-Witt invariant of trace forms; we sketch this proof. For $\zeta_n$ a primitive $n$th root of unity, we set $\xi_{n,j}:=\zeta^j_n+\zeta_n^{-j}$ and $\eta_{n,j}=\zeta^j_n-\zeta^{-j}_n$ and set  $\xi_n=\xi_{n,1}$. Let $Q_p$ be the trace form $\Tr_{\Q(\xi_p)/\Q}$.  Using Serre's theorem and the fact that the inclusion  $\Q(\xi_p)\subset \Q(\zeta_p)$ solves the embedding problem for $\Q(\xi_p)$ \cite[\S 3.1.3]{Serre}, we find  
\[
w_2(Q_p)=(2,p^{\frac{p-3}{2}})=\begin{cases} 0&\text{ for } p\equiv 3\mod 4\\
(2,p)&\text{ for } p\equiv 1\mod 4.
\end{cases}
\]
Since $\Q(\xi_p)$ is totally real, $Q_p$ has signature equal to its rank $(p-1)/2$; by comparing the complete set of invariants: rank signature, discriminant, Hasse-Witt invariant (see e.g.\cite[Chap. 5, Theorem 3.4 and Theorem 6.4]{Scharlau}), we have the equivalence of quadratic forms over $\Q$:
\[
Q_p\sim \begin{cases} \frac{p-1}{2}\<1\>&\text{ for }p\equiv 3\mod 4\\
\<2\>+\<2p\>+\frac{p-5}{2}\<1\>&\text{ for }p\equiv 1\mod 4
\end{cases}
\]

We now compute the discriminant and Hasse-Witt invariant for the trace form $Q_{4p}:=\Tr_{\Q(\xi_{4p})/\Q}$. We have the $\Q$-basis of $\Q(\xi_p)$, 
\[
\xi_{p,1},\ldots, \xi_{p,d}; \quad d=\frac{p-1}{2}.
\]
We extend this to the basis
\[
\xi_{p,1},\ldots, \xi_{p,d}, i\eta_{p,1},\ldots, i\eta_{p,d}
\]
for $\Q(\xi_{4p})$. Letting $V_1$ be the $\Q$-subspace spanned by the $\xi_{p,j}$ und $V_2$ the $\Q$-subspace spanned by the $i\eta_{p,j}$, it follows from a direct computation that $V_1$ and $V_2$ are orthogonal with respect to $Q_{4p}$, that $Q_{4p}$ restricted to $V_1$ is equal to $2\cdot Q_p$ and that the restriction of $Q_{4p}$ to $V_2$ is (in this basis) the diagonal form $\frac{p-1}{2}\cdot\<2p\>$. In other words 
\[
Q_{4p}\sim \begin{cases} d\cdot\<2\>+d\cdot\<2p\>
&\text{ for }p\equiv 3\mod 4\\
\<1\>+\<p\>+(d-2)\cdot \<2\>+d\cdot\<2p\>&\text{ for }p\equiv 1\mod 4
\end{cases}
\]

Using this, one computes the discriminant and Hasse-Witt invariant of $Q_{4p}+\<p\>$ and finds the discriminant is 1 modulo squares, and the Hasse-Witt invariant is zero. Since $\Q(\xi_{4p})$ is totally real, $Q_{4p}+\<p\>$ has signature $p$. This shows that 
$Q_{4p}+\<p\>\sim p\<1\>$. 
\end{proof}

\section{Characteristic classes of symmetric powers} \label{sec:SymPower}

For a positive integer $m$ we write $m!!$ for the product 
\[
m!!=\prod_{i=0}^{[m/2]} m-2i.
\]

For $F=ke_1\oplus ke_2$, we give $\Sym^m F$ the oriented basis $e_1^m, e_1^{m-1}e_2$, $\ldots$, $e_1e_2^{m-1}, e_2^m$.  Given a two-dimensional representation $\rho:G\to \SL_2$, this defines the representation $\Sym^m\rho:G\to \SL_{m+1}$.  
 
\begin{theorem}\label{thm:SymRnk2} Let  $m$ be a positive integer, $k$ a perfect field. Let $E\to X$ be a rank two vector bundle on some $X\in \Sm/k$ and consider the Euler class $e(\Sym^mE)\in H^{m+1}(X,\sW(\det^{-1} \Sym^mE))$. Then for $m=2r+1$ odd, and  $k$ of characteristic 0, or of positive characteristic $\ell$ prime to $2m$, we have
\[
e(\Sym^mE)=m!!\cdot e(E)^{r+1},
\]
For $k$ an arbitrary field of characteristic $\neq2$ and for  $m$ even, we have $e(\Sym^mE)=0$.

The total Pontryagin class $p:=1+\sum_{i\ge 1}p_i$ is given by
\[
p(\Sym^mE)=\prod_{i=0}^{[m/2]}(1+(m-2i)^2e(E)^2)
\]
\end{theorem}

\begin{proof} We first compute the Euler class. The bundle $\Sym^mE$ has rank $m+1$, so has odd rank if $m$ is even, and therefore has vanishing Euler class in Witt cohomology by \cite[Proposition 7.3]{LevEnum} or Lemma~\ref{lem:PStab}.  We consider the case of odd $m=2r+1$.

It clearly suffices to prove the formula for the universal bundle $E_2\to \BGL_2$. By Theorem~\ref{thm:BSLDecomp}, it suffices to prove the formula for the universal bundle $E_2\to \BSL_2$ and then by Proposition~\ref{prop:WittCoh1}, we reduce to the case of the rank two bundle $\tilde{E}_2=\tilde{O}(1)\to BN$. 

It is easy to see that we have an orientation-preserving isomorphism
\[
\Sym^m\rho_1\cong \oplus_{i=0}^r \rho^{(-1)^i}_{m-2i}
\]
where we write $\rho_n^{+1}$ for $\rho_n$ and $\rho_n^{-1}$ for $\rho_n^-$. This
 gives the isomorphism
\[
\Sym^m\tilde{O}(1)\cong \oplus_{i=0}^r\tilde{O}^{(-1)^i}(m-2i)
\]
and thus by Theorem~\ref{thm:EulerClassOm} we have
\[
e(\Sym^m\tilde{O}(1))=\prod_{i=0}^r e(\tilde{O}^{(-1)^i}(m-2i))=
\begin{cases} \prod_{i=0}^r (m-2i)\cdot p^*e &\text{ for }r\text{ even,}\\
\prod_{i=0}^r (-1)\cdot(m-2i)\cdot p^*e &\text{ for }r\text{ odd.} 
\end{cases}
\]
which in either case gives 
\[
e(\Sym^m\tilde{O}(1))
=m!!\cdot p^*e^{r+1}.
\]

For the total Pontryagin class, we recall that $p_i$ lives in $H^{4i}(-,\sW)$, and for a rank 2 bundle $F$, $p_1(F)=e(F)^2$ and the higher Pontryagin classes vanish. For a rank 1 bundle $F$, $p(F)=1$, and for arbitrary bundles $F, F'$, $p$ satisfies the Whitney formula
\[
p(F\oplus F')=p(F)\cup p(F').
\]

Thus, for $m=2r$ even, we have 
\[
\Sym^m\rho_1\cong\begin{cases}  \rho_0\oplus \oplus_{i=0}^{r-1} \rho_{m-2i}&\text{ for }r\text{ even }\\
\rho_0^-\oplus \oplus_{i=0}^{r-1} \rho_{m-2i}&\text{ for }r\text{ odd.}
\end{cases}
\]
which by Theorem~\ref{thm:EulerClassOm} gives
\[
p(\Sym^m\tilde{O}(1))=\prod_{i=0}^r(1+(m-2i)^2p^*e^2).
\]
For  $m=2r+1$ odd, we have as above
\[
\Sym^m\rho_1\cong  \oplus_{i=0}^r\rho^{\pm}_{m-2i}
\]
which by Theorem~\ref{thm:EulerClassOm}  again gives
\[
p(\Sym^m\tilde{O}(1))=\prod_{i=0}^r(1+(m-2i)^2p^*e^2).
\]
\end{proof}

\begin{ex} Consider a smooth hypersurface $X$ in $\P^{d+1}_k$ of degree $2d-1$ with $k$ a perfect field of characteristic prime to $2(2d-1)$. Counting constants, one expects $X$ to contain finitely many lines. In fact, if $X$ is defined by a degree $2d-1$ homogeneous polynomial $f\in k[T_0,\ldots, T_{d+1}]$, then $f$ defines a global section $s_f$ of $\Sym^{2d-1}E^\vee_2$, where $E_2^\vee$ is the dual of the tautological rank 2 subbundle $E_2\subset \sO_{\Gr}^{d+2}$ on the Grassmannian $\Gr=\Gr(2, d+2)$, and a line $\ell\subset \P^{d+1}$ is contained in $X$ exactly when $s_f([\ell])=0$, where $[\ell]\in \Gr$ is the point corresponding to the 2-plane $\pi\subset \A^{d+2}$ with $\ell=\P(\pi)$. Since $\Gr$ has dimension $2d=\rnk \Sym^{2d-1}E^\vee_2$, one would expect that for a general $f$, $s_f$ would  have finitely many zeros.

In the case of our section $s_f$, with $X$ having finitely many lines, we may apply our Theorem~\ref{thm:SymRnk2} to find 
\begin{multline*}
\sum_{x,\ s_f(x)=0}e_x(\Sym^{2d-1}E_2^\vee;s_f)=
e(\Sym^{2d-1}E_2^\vee)\\=(2d-1)!!e(E_2^\vee)^d\in H^{2d}(\Gr, \sW(\sO_\Gr(-d))
\end{multline*}
 Here $e_x(\Sym^{2d-1}E_2^\vee;s_f)$ is the image of the local index $s_f^*(\th(\Sym^{2d-1}E_2^\vee))\in 
H^{2d}_x(\Gr, \sW(\sO_\Gr(-d))$ in $H^{2d}(\Gr, \sW(\sO_\Gr(-d))$. Note that  $\det E_2^\vee=\sO_{\Gr}(1)$ (with respect to the Pl\"ucker embedding).

The canonical exact sequence presenting the tangent bundle
\begin{equation}\label{eqn:TanBun}
0\to E_2\otimes E_2^\vee\to \oplus_{i=0}^{d+1}E_2^\vee\cdot e_i \xrightarrow{\pi} T_{\Gr/k}\to 0
\end{equation}
and the canonical isomorphism $\det  E_2\otimes E_2^\vee\cong \sO_\Gr$ gives the isomorphism $\omega_{\Gr/k}=\sO_{\Gr}(-d-2)$, which in turn gives the isomorphism
\[
\sK^{MW}_*(\sO_\Gr(-d))\cong \sK^{MW}_*(\omega_{\Gr/k}).
\]
The push-forward  
\[
\pi_{\Gr*}: H^d(\Gr, \sK^{MW}_d(\omega_{\Gr/k}))\to H^0(\Spec k, \sGW)=\GW(k)
\]
thus gives us well-defined push-forward maps
\begin{align*}
&\pi_{\Gr*}: H^d(\Gr, \sK^{MW}_d(\sO_\Gr(-d)))\to \GW(k),\\
&\pi_{\Gr*}: H^d(\Gr, \sW(\sO_\Gr(-d)))\to W(k).
\end{align*}
In addition, we claim that
\[
\pi_{\Gr*}(e(E_2^\vee)^d)=\<1\>\in \GW(k). 
\]
Indeed,  a linear form $L$ on $\P^{d+1}$ gives a section $s_L$ of $E_2^\vee$, transverse to the 0-section,  with zero-locus $z_L$ exactly the Grassmannian of lines contained in the hyperplane $L=0$. Taking $L$ to be the coordinate $T_i$, $i=0,\ldots, d-1$, the corresponding (transverse) intersection of 0-loci is the point of $\Gr$ corresponding to the  line $\ell_0$ defined by $T_0=\ldots=T_{d-1}=0$. $\ell_0$ is a point of the big affine cell $\Gr^0\subset \Gr$ consisting of the lines $\ell$ that do not intersect the linear subspace $T_d=T_{d+1}=0$. For $\ell\in \Gr^0$, $\ell$ intersects the hyperplane $T_{d+1}=0$ in a single point $(x_0,\ldots, x_{d-1}, 1,0)$ and the 
the hyperplane $T_{d}=0$ in a single point $(y_0,\ldots, y_{d-1}, 0,1)$, giving the coordinates $(x_0, y_0, x_1, y_1,\ldots, x_{d-1}, y_{d-1})$ for $\Gr^0\cong \A^{2d}$. The sections $s_{T_d}$, $s_{T_{d+1}}$ of $E_2^\vee$ give a framing trivializing $E_2^\vee$ over $\Gr^0$ and composing the map $\pi$  \eqref{eqn:TanBun} with the inclusion of subbundle $\oplus_{i=0}^{d-1}E_2^\vee\cdot e_i$ of $\oplus_{i=0}^{d+1}E_2^\vee\cdot e_i$ defines an isomorphism $\alpha:\oplus_{i=0}^{d-1}E_2^\vee\cdot e_i\to T_{\Gr/k}$ over $\Gr^0$. 

The class $e(E_2^\vee)^d$ is represented in $H^d_{\ell_0}(\Gr, \sK^{MW}(\sO(-d))$ by the pullback of $\th(\oplus_{i=0}^{d-1}E_2^\vee\cdot e_i)$ by the section $s:=\sum_{i=0}^{d-1}s_{T_i}e_i$. We can identify  $H^d_{\ell_0}(\Gr, \sK^{MW}(\det^{-d} E_2^\vee))$ with
$\GW(k(\ell_0);  \det^{-d} E_2^\vee\otimes \omega_{\Gr/k}^{-1}\otimes k(\ell_0))$ via the Rost-Schmid resolution. Since $s$ is transverse to the zero-section, the  element of $\GW(k(\ell_0);  \det^{-d} E_2^\vee\otimes \omega_{\Gr/k}^{-1})$ corresponding to  $s^*(\th(\oplus_{i=0}^{d-1}E_2^\vee\cdot e_i))$ is given by the rank one form $\<\det^{-1} \bar{ds}\>$, where $\bar{ds}:T_{\Gr, \ell_0}\to \oplus_{i=0}^{d-1}E_2^\vee\cdot e_i$ is the composition of the differential $ds:T_{\Gr, \ell_0}\to T_{\ell_0,0}(\oplus_{i=0}^{d-1}E_2^\vee\cdot e_i)$ with the projection $T_{\ell_0,0}(\oplus_{i=0}^{d-1}E_2^\vee\cdot e_i)\to  \oplus_{i=0}^{d-1}E_2^\vee\cdot e_i$ with kernel the tangent space to the 0-section. A direct computation shows that $\bar{ds}$ is exactly the inverse to $\alpha\otimes k(\ell_0)$, so after using \eqref{eqn:TanBun} and our trivialization of $E_2^\vee$ to identify $\det^{-d} E_2^\vee\otimes \omega_{\Gr/k}^{-1}\otimes k(\ell_0)$ with $k(\ell_0)$ we obtain the form $\<1\>\in \GW(k(\ell_0))$; as $k=k(\ell_0)$, this pushes forward to $\<1\>\in \GW(k)$, as claimed. 

The computation above yields the class of $\pi_{\Gr*}(e(\Sym^{2d-1}(E_2^\vee))\in W(k)$ as
\[
\pi_{\Gr*}(e(\Sym^{2d-1}(E_2^\vee))=(2d-1)!!\<1\>\in W(k).
\]
We can lift this to an identity in $\GW(k)$ by computing the rank, which is just the degree of the top Chern class of $\Sym^{2d-1}E_2^\vee$, that is, the number of lines on $X$ over an algebraically closed field, say $N_d$. This gives
\[
\pi_{\Gr*}(e(\Sym^{2d-1}(E_2^\vee))=(2d-1)!!\<1\>+\frac{N_d-(2d-1)!!}{2}\cdot (\<1\>+\<-1\>)\in \GW(k).
\]
For $k=\R$, this recovers a computation of Okonek-Teleman \cite{OkonekTeleman} by taking the signature.

For $d=2$, $N_d=27$ and we have
\[
\pi_{\Gr*}(e(\Sym^{3}(E_2^\vee))=3\<1\>+12\cdot (\<1\>+\<-1\>)=15\<1\>+12\<-1\>\in \GW(k),
\]
recovering a computation of Kass-Wickelgren \cite{KW2}.

Additionally, Kass-Wickelgren  \cite{KW1} have shown how to  compute  explicitly the local contribution $e_x(V;s)$  to the global Euler class $e(V)$ of an isolated zero of a section $s$ of a vector bundle $V\to Y$ of rank equal to the dimension $d$ of $Y$. In case $s$ has only isolated zeros, this gives the formula
\[
e(V)=\sum_{s(x)=0}e_x(V;s).
\]
They give a geometric interpretation of this local contribution in the case of the section $s_f$ for $d=2$, and use their computation of the pushforward of the Euler class to deduce  a mod 2 congruence for the number of lines of different geometric type defined over even and odd extensions of the base field $k$ for $k$ a finite field.  It would be interesting to extend this geometric interpretation to the higher dimensional case.
\end{ex}

\section{Tensor products and tensor powers}\label{sec:TensorPower}

In this section we assume $\Char k\neq 2$. To discuss characteristic classes of tensor products, we need to fix a canonical isomorphism $(\det E)^{\rnk F}\otimes (\det F)^{\rnk E}\cong \det (E\otimes F)$. It suffices to do this for free $R$-modules $E=\oplus_{i=1}^nRe_i$, $F=\oplus_{j=1}^mRf_j$, functorially in $R$, which we do by sending
\[
(e_1\wedge\ldots\wedge e_n)^{\otimes m}\otimes (f_1\wedge\ldots\wedge f_m)^{\otimes n}
\]
to
\[
(e_1\otimes f_1)\wedge\ldots\wedge(e_n\otimes f_1)\wedge\ldots\wedge
(e_1\otimes f_m)\wedge\ldots\wedge(e_n\otimes f_m)
\]

Similarly, we have the canonical isomorphism $\det E\otimes \det F\cong \det(E\oplus F)$, determined by sending $(e_1\wedge\ldots\wedge e_n)\otimes(f_1\wedge\ldots f_m)$ to $((e_1,0)\wedge\ldots\wedge (e_n,0))\wedge ((0,f_1)\wedge\ldots (0,f_m))$ in the case $E=\oplus_{i=1}^nRe_i$, $F=\oplus_{j=1}^mRf_j$.

\begin{proposition} \label{prop:TensorRank2} Let $E\to X$, $E'\to Y$ be rank two bundles on $X, Y\in \Sm/k$ with trivialized determinants.  Then
\begin{align*}
&e(\pi_1^*E\otimes \pi_2^*E')=\pi_1^*e(E)^2-\pi_2^*e(E')^2\in H^4(X\times Y,\sW)\\
&p_1(\pi_1^*E\otimes \pi_2^*E')=2(\pi_1^*e(E)^2+\pi_2^*e(E')^2)\in H^4(X\times Y,\sW)\\
&p_2(\pi_1^*E\otimes \pi_2^*E')=(\pi_1^*e(E)^2-\pi_2^*e(E')^2)^2\in H^8(X\times Y,\sW)
\end{align*}
and $p_m(\pi_1^*E\otimes \pi_2^*E')=0$ for $m>2$.
\end{proposition}

\begin{proof} The last formula follows from the fact that $p_m(F)=0$ for $2m>\rnk(F)$. The formula for $p_2$ follows from the formula for $e$ and the fact that $p_m(F)=e(F)^2$ for $F$ of rank $2m$.

It suffices to consider the universal case of $\pi_1^*\tilde{E}_2\oplus \pi_2^*\tilde{E}_2$ on $\BSL_2\times \BSL_2$.  Let 
\[
i_1:\BSL_2\to \BSL_2\times \BSL_2,\  i_2:\BSL_2\to \BSL_2\times \BSL_2
\]
be the inclusions $i_1(x)=(x\times x_0)$, $i_2(x)=x_0\times x$, where $x_0$ is the base-point, and let $\delta:\BSL_2\to \BSL_2\times \BSL_2$ be the diagonal inclusion. We have
\begin{align*}
&i_1^*\pi_1^*\tilde{E}_2\otimes \pi_2^*\tilde{E}_2=\tilde{E}_2\otimes \tilde{E}_{2x_0}\\
&i_2^*\pi_1^*\tilde{E}_2\otimes \pi_2^*\tilde{E}_2=\tilde{E}_{2x_0}\otimes\tilde{E}_2\\
&\delta^*\pi_1^*\tilde{E}_2\otimes \pi_2^*\tilde{E}_2\cong \Sym^2\tilde{E}_2\oplus O_{\BSL_2}
\end{align*}
We have a canonical isomorphism $\tilde{E}_{2x_0}\cong k^2$, which gives us isomorphisms
\begin{align*}
&i_1^*\pi_1^*\tilde{E}_2\otimes \pi_2^*\tilde{E}_2\cong \tilde{E}_2\oplus \tilde{E}_{2}\\
&i_2^*\pi_1^*\tilde{E}_2\otimes \pi_2^*\tilde{E}_2\cong\tilde{E}_{2}\oplus \tilde{E}_2
\end{align*}
Furthermore, the canonical isomorphism
\[
\det\tilde{E}_2\cong O_{\BSL_2}
\]
induces the isomorphisms
\begin{align*}
&\alpha:\det(\pi_1^*\tilde{E}_2\otimes \pi_2^*\tilde{E}_2)\to O_{\BSL_2\times \BSL_2}\\
&\beta: \det(\tilde{E}_2\oplus \tilde{E}_{2})\to O_{\BSL_2}
\end{align*}
using our convention for the determinant of a tensor product and a direct sum, as described above. This gives us the identities
\[
i_1^*\alpha=\beta,\ i_2^*\alpha=-\beta;
\]
in terms of free modules $E=Re_1\oplus Re_2$, $F=Rf_1\oplus Rf_2$, this arises from the fact that 
\[
e_1\otimes f_1, e_2\otimes f_1, e_1\otimes f_2, e_2\otimes f_2
\]
is an oriented basis for $E\otimes F$, giving the oriented basis
\[
e_1\otimes f_1, e_2\otimes f_1, e_1\otimes f_2, e_2\otimes f_2
\]
for $E\otimes f_1\oplus E\otimes f_2$, while an oriented basis for $e_1\otimes F\oplus e_2\otimes F$ is
\[
e_1\otimes f_1, e_1\otimes f_2, e_2\otimes f_1, e_2\otimes f_2,
\]
which introduces the sign -1  in the identity $i_2^*\alpha=-\beta$.

Keeping track of the orientations gives the identities
\begin{align*}
i_1^*e(\pi_1^*\tilde{E}_2\otimes \pi_2^*\tilde{E}_2)&=e(\tilde{E}_2\oplus \tilde{E}_2)=e(\tilde{E}_2)^2\\
i_2^*e(\pi_1^*\tilde{E}_2\otimes \pi_2^*\tilde{E}_2)&=-e(\tilde{E}_2\oplus \tilde{E}_2)=-e(\tilde{E}_2)^2\\
\delta^*(\pi_1^*\tilde{E}_2\otimes \pi_2^*\tilde{E}_2)&=e(\Sym^2\tilde{E}_2)\cup
e(O_{\BSL_2})=0
\end{align*}

On the other hand, we have
\[
H^*(\BSL_2\times \BSL_2, \sW)=W(k)[\pi_1^*e, \pi_2^*e]
\]
so there are uniquely defined elements $a,b,c\in W(k)$ with 
\[
e(\pi_1^*\tilde{E}_2\otimes \pi_2^*\tilde{E}_2)=a\pi_1^*e^2+b\pi_1^*e\pi_2^*e+c\pi_2^*e^2.
\]
The computations above together with the identities
\[
i_t^*\pi_s^*e=0\text{ for }s\neq t,\ i_t^*\pi_t^*e=e,\ \delta^*\pi_t^*e=e,
\]
imply that $a=1$,   $c=-1$ and   $b=0$.

For the Pontryagin class, we recall that the Pontryagin classes are stable and are insensitive to the orientation. This gives  
\[
i_1^*p(\pi_1^*\tilde{E}_2\otimes \pi_2^*\tilde{E}_2)=i_2^*p(\pi_1^*\tilde{E}_2\otimes \pi_2^*\tilde{E}_2)=
p(\tilde{E}_2)^2=(1+e(\tilde{E}_2)^2)^2
\]
and using theorem~\ref{thm:SymRnk2} to compute $p_1(\Sym^2\tilde{E}_2)$ gives
\[
\delta^*p(\pi_1^*\tilde{E}_2\otimes \pi_2^*\tilde{E}_2)
=1+4e(\tilde{E}_2)^2
\]
Computing as above gives
\[
p_1(\pi_1^*\tilde{E}_2\otimes \pi_2^*\tilde{E}_2)=2(\pi_1^*e^2+\pi_2^*e^2)
\]
\end{proof}

\begin{remark} Combining Ananyevskiy's $\SL_2$ splitting principle with Theorem~\ref{thm:BSLDecomp} reduces the computation of the Euler class and Pontryagin classes in Witt cohomology of any Schur functor applied to a vector bundle to the case of a direct sum of rank 2 bundles with trivialized determinant. Via the representation theory of $\SL_2$, this reduces the computation to  the case of tensor products of symmetric powers. Using the $\SL_2$ splitting principle again, we need only know the characteristic classes of symmetric powers and the tensor product of two rank two bundles (with trivialized determinant); as this information is given by Theorem~\ref{thm:SymRnk2} and Proposition~\ref{prop:TensorRank2}, we have  a complete calculus for the computation of the characteristic classes, at least over a field of characteristic zero.

In positive characteristic, our calculus is limited by the restriction on the characteristic in Theorem~\ref{thm:SymRnk2}, but one still has a partial calculus for the characteristic classes of Schur functors of degree $\le m$ in characteristics $\ell>m$. 
\end{remark}

\end{document}